\documentclass[12pt,a4]{amsart}
 \usepackage[a4paper, left=28mm, right=28mm, top=26mm, bottom=31mm]{geometry}
 \usepackage{amsthm,amssymb}
 \usepackage[all]{xy}
 \usepackage{amsmath}
 \usepackage{textcomp}
 \usepackage{cases}
 \usepackage{here}
 \usepackage[dvipdfmx]{graphicx}
 \usepackage{amscd}
 \usepackage{mathrsfs}
 \usepackage{float}
 \usepackage{wrapfig}
 \usepackage{color}
 \usepackage{cleveref} 
 \usepackage{appendix}

\newtheorem{thm}{Theorem}[section]
 \newtheorem{dfn}[thm]{Definition}
 \newtheorem{lem}[thm]{Lemma}
 \newtheorem{prp}[thm]{Proposition}
 \newtheorem{rem}[thm]{Remark}
 
 \newtheorem{cor}[thm]{Corollary}

 \newtheorem{exam}[thm]{Example}
 \def\P{\mathbb{P}}
 \def\A{\mathbb{A}}
 \def\G{\mathbb{G}}
 \def\R{\mathbb{R}}
 \def\Z{\mathbb{Z}}
 \def\Q{\mathbb{Q}}

 \def\F{\mathscr{F}}
 \def\C{\mathscr{C}}

 \def\X{\mathscr{X}}
 
\def\rec{\mathop{\mathrm{rec}}\nolimits}
 \def\Shv{\mathop{\mathrm{Shv}}\nolimits}
 \def\AP{\mathop{\mathrm{AP}}\nolimits}

 \def\sgn{\mathop{\mathrm{sgn}}\nolimits}

 \def\Conv{\mathop{\mathrm{Conv}}\nolimits}
 \def\cone{\mathop{\mathrm{cone}}\nolimits}
 \def\cyl{\mathop{\mathrm{cyl}}\nolimits}
 
 \def\Res{\mathop{\mathrm{Res}}\nolimits}

 \def\Tan{\mathop{\mathrm{Tan}}\nolimits}
 
 \def\Zar{\mathop{\mathrm{Zar}}\nolimits}

 \def\Id{\mathop{\mathrm{Id}}\nolimits}
 \def\CH{\mathop{\mathrm{CH}}\nolimits}
 
 \def\Gr{\mathop{\mathrm{Gr}}\nolimits}
 
 \def\res{\mathop{\mathrm{res}}\nolimits}
 \def\Im{\mathop{\mathrm{Im}}\nolimits}
 \def\Ker{\mathop{\mathrm{Ker}}\nolimits}
 
 \def\sing{\mathop{\mathrm{sing}}\nolimits}

 \def\Trop{\mathop{\mathrm{trop}}\nolimits}
 
 \def\relint{\mathop{\mathrm{rel.int}}\nolimits}
 
 \def\Star{\mathop{\mathrm{Star}}\nolimits}

 \def\Spec{\mathop{\mathrm{Spec}}\nolimits}

 \def\Hom{\mathop{\mathrm{Hom}}\nolimits}

 \def\id{\mathop{\mathrm{id}}\nolimits}
 \def\Trop{\mathop{\mathrm{Trop}}\nolimits}

 \def\Gr{\mathop{\mathrm{Gr}}\nolimits}
 
\numberwithin{equation}{section}
 \numberwithin{figure}{section}

 \makeatletter
 \DeclareRobustCommand{\genericinterval}[2]{%
  \@ifstar{\genericinterval@star{#1}{#2}}{\genericinterval@nostar{#1}{#2}}}
   \newcommand{\genericinterval@star}[4]{\mathopen{}\mathclose{\left#1#3,#4\right#2}}
    \newcommand{\genericinterval@nostar}[4]{\mathopen{#1}#3,#4\mathclose{#2}}

         \makeatother
 
\begin{document}
 \title[Tropical cohomology via reductions]{Tropical cohomology via reductions of tropical varieties}
 \author{Ryota Mikami}
 \address{Institute of Mathematics, Academia Sinica, Astronomy-Mathematics Building, No.\ 1, Sec.\ 4, Roosevelt Road, Taipei 10617, Taiwan.}
 \email{ryotamikamimath467jhoiv9dhk3@gmail.com}
 \date{\today}
 
\begin{abstract}
  Itenberg-Katzarkov-Mikhalkin-Zharkov 
  gave an isomorphism of tropical cohomology and cohomology of some maximally degenerate algebraic varieties.
  Their proof was based on tropical analogs of Steenbrink's geometric monodromy-weight spectral sequences.
  These were generalized to the non-realizable case
  by  
  Amini-Piquerez.  
    In this paper, we give a new construction of these tropical spectral sequences in the same way as Steenbrink's ones. 
  For this purpose, we introduce 
    reductions of tropical varieties. 
  We also show that eigenwave actions are given by 
   tropical Gauss-Manin connections.
 \end{abstract}
 
 \maketitle
 \setcounter{tocdepth}{1}
 \tableofcontents

\section{Introduction}\label{sec:introduction}
  \emph{Tropical cohomology} $H^{p,q}_{\Trop} (W)$ for a tropical variety $W$ 
   was introduced by 
   Itenberg-Katzarkov-Mikhalkin-Zharkov \cite{ItenbergKatzarkovMikhalkinZharkovTropicalhomology2019}. 
   They proved that for a one parameter 
   family 
   $\mathcal{Z} = (Z_t)_{t \in \mathcal{D}^*} \subset \P^n \times \mathcal{D}^{*}$ 
   of smooth complex projective varieties 
   over a punctured disk $\mathcal{D}^*$
   with a smooth (i.e., locally matroidal) tropicalization $\Trop (\mathcal{Z}_{\mathbb{C} ((t))})$, 
   we have 
   \begin{align}\label{eq tro coh = LMHS}
   H^{p,q}_{\Trop}(\Trop (\mathcal{Z}_{\mathbb{C} ((t))}))
   \cong 
   \Gr_{2p}^W H^{p+q} (Z_{\infty}),
   \end{align}
   where $\Gr_{2p}^W H^{p+q} (Z_{\infty})$ is
   the weight graded quotient
   of the limit mixed Hodge structure.
   Since, in this case,   
   the odd weight graded quotients vanish 
   and $\Gr_{2p}^W H^{p+q} (Z_{\infty})$ is of pure $(p,p)$-type, 
   hodge numbers of general members $Z_t$
   equal tropical hodge numbers of $
    \Trop(\mathcal{Z}_{\mathbb{C}((t))})$.
   (Gross-Siebert \cite{GrossSiebertMirrorsymmetryvialogarithmicdegenerationdataII2010} also gave a similar cohomology and isomorphism for some integral affine manifolds with singularities.)

  The proof in \cite{ItenbergKatzarkovMikhalkinZharkovTropicalhomology2019} 
  was based on two spectral sequences.
   One is Steenbrink's geometric monodromy-weight spectral sequence (\cite[Corollary 4.20]{SteenbrinkLimitsofHodgestructures1976})
   $$ E_1^{p,q} 
      = \bigoplus_{ \max \{ 0, p \} \leq k }
        \bigoplus_{\substack{J \subset I \\ \# J = -p +2k +1}}
      H^{p+2q-2k}_{\sing} \big( \bigcap_{i \in J}  Z_i  \big) (p-k) 
      \Rightarrow 
      H^{p+q} (Z_{\infty}), 
   $$ 
    where  $Z_0 = \bigcup_{i \in I} Z_i $ is 
    a strictly semi-stable reduction 
    of $\mathcal{Z}$.
   The other one, given in \cite{ItenbergKatzarkovMikhalkinZharkovTropicalhomology2019},
   is its tropical analog.
   This tropical analog  conveges to tropical cohomology, and has $E_1$-terms isomorphic to Steenbrink's one.
   Hence we get isomorphism (\ref{eq tro coh = LMHS}).

     Amini-Piquerez \cite{AminiPiquerezHodgetheoryfortropicalvarieties2020} 
    generalized 
   this tropical spectral sequence to 
   abstract compact tropical homology manifolds $X $
   (i.e., tropical varieties regular at infinity 
    satisfying Poincar\'{e}-Verdier duality) 
    under certain assumption (see \Cref{introduction spectral sequence for ss reduction} and \Cref{intro the existence of unimodular triangulation}).
     As an corollary, 
  isomorphism (\ref{eq tro coh = LMHS}) was generalized to  
   more general maximally degenerate families in $\P^n$
   by Aksnes-Amini-Piquerez-Shaw  
   \cite{AksnesAminiPiquerezShawCOHOMOLOGICALLYTROPICALVARIETIES2025}.

  The construction of tropical spectral sequences
   in \cite{ItenbergKatzarkovMikhalkinZharkovTropicalhomology2019} and 
   \cite{AminiPiquerezHodgetheoryfortropicalvarieties2020}
   is conceptually similar to, but technically different from that of Steenbrink's one (\cite[Corollary 4.20]{SteenbrinkLimitsofHodgestructures1976}).
  While Steenbrink's one 
    is based on  
    \emph{relative log holomorphic differential forms} 
    (cf. Deligne's work on mixed Hodge theory \cite{DeligneThoriedeHodgeII1972}), 
  the tropical one 
   is based on explicit double complexes using polyhedral complex structures of tropical varieties,
   and 
   uses 
     the exactness of certain complexes. 
     In particular, 
     this exactness 
     requires 
    tropical Poincar\'{e}-Verdier duality, 
    which is close to 
    ``maximally degenerate'' assumption in 
     complex algebraic geometry.
     Of course, 
    such an assumption is not required in Steenbrink's one.

  The main goal of this paper is 
  to give a simpler construction of 
   the tropical spectral sequences 
   in the same way as Steenbrink 
   (without tropical Poincar\'{e}-Verdier duality).

  \begin{thm}[\Cref{spectral sequence for ss reduction}, 
  \cite{ItenbergKatzarkovMikhalkinZharkovTropicalhomology2019}
   for smooth realizable $X$,
   \cite{AminiPiquerezHodgetheoryfortropicalvarieties2020} for a tropicl homology manifold $X$]
  \label{introduction spectral sequence for ss reduction}
   Let $X$ be a ($\Q$-rational) tropical variety regular at infinity (\Cref{def:regular at infinity}) in a smooth proper tropical toric variety $\Trop (T_{\Sigma})$ with cochacter lattice $N$.
   We assume that $X$ has a $N$-unimodular 
    $\Q$-rational polyhedral complex structure $\Lambda$ (\Cref{def unimodularity:label}).
    Let $X_{\infty}:= \bigcup_{i \in I} X_i$ 
    be 
    the strictly semi-stable reduction (\Cref{sub:semi_stable_reductions_of_tropical_varieties}) of $X$
    given by $\Lambda$.

    Then
    for $r \geq 0$, 
    there exists a spectral sequence 
    $$E_1^{p,q}= 
    \bigoplus_{ \max \{0,p\} \leq u -r }
    \bigoplus_{\substack{ J \subset I \\ \# J = -p + 2 (u-r) +1 } }
  H^{2r-u+p, p+q -u}_{\Trop} \big(\bigcap_{i \in J} X_i \big)
     \Rightarrow H^{r,p+q-r}_{\Trop } (X)  , $$
    whose differential $d_1 \colon E_1^{p,q} \to E_1^{p+1,q}$ is given by pullback maps and Gysin maps, up to signs (see \Cref{differential of monodromy-weight spectral sequence:label} for the signs).
  \end{thm}

  \begin{rem}
   In the case of \cite{ItenbergKatzarkovMikhalkinZharkovTropicalhomology2019} and 
   \cite{AksnesAminiPiquerezShawCOHOMOLOGICALLYTROPICALVARIETIES2025} (which uses 
   \cite{AminiPiquerezHodgetheoryfortropicalvarieties2020}),
    tropical cohomology of $\bigcap_{i \in J} X_i $
  is isomorphic to singular cohomology of 
  the corresponding stratum closure $\bigcap_{i \in J} Z_i $ of the semi-stable reduction $Z_0$.
  Hence we get isomorphism (\ref{eq tro coh = LMHS}).
  See [loc.cit.] for details (and also above \Cref{cohomology comparison in sm projective toric varieties:label}).
  \end{rem}

  In \cite{ItenbergKatzarkovMikhalkinZharkovTropicalhomology2019} 
  and \cite{AminiPiquerezHodgetheoryfortropicalvarieties2020}, 
    tropical cohomology $
  H^{2r-u+p, p+q -u}_{\Trop} (\bigcap_{i \in J} X_i )$
  of closure strata $\bigcap_{i \in J} X_i $ 
  were considered 
  without defining 
   the \emph{strictly semi-stable reduction} $X_{\infty}$.
  This reduction $X_{\infty}$
  is the point of our proof. 
  It 
  enable us to give 
  an analog of Steenbrink's construction.
   This reduction of a tropical variety is an abstraction of tropicalizations (with respect to trivial valuations) of 
   strictly semi-stable reductions of algebraic varieties,
   and is
   given by taking a special fiber of a compactification of 
   a cone $\overline{ \R_{>0 } \cdot ( X\times \{ 1  \} )} $ 
   (see \Cref{sub:semi_stable_reductions_of_tropical_varieties} for details).
  This cone itself 
   is used in \cite{ItenbergKatzarkovMikhalkinZharkovTropicalhomology2019}, 
   and is also studied systematically 
   by Gubler \cite{GublerAguidetotropicalizations2013}
   to introduce tropical compactifications 
   in the non-trivially valued case.
   Strictly semistable reductions of tropical varieties 
    might also be known to some specialists. 
    However, to the author's knowledge, there seems to be no explicit reference in the literature.  

  The $N$-unimodular  
    $\Q$-rational polyhedral complex structure $\Lambda$
    of the tropical variety $X$ 
    is needed to construct the reduction $X_{\infty}$.
    Such a structure is also needed in 
  \cite{ItenbergKatzarkovMikhalkinZharkovTropicalhomology2019}
  and 
   \cite{AminiPiquerezHodgetheoryfortropicalvarieties2020}. 
  In \cite{ItenbergKatzarkovMikhalkinZharkovTropicalhomology2019},
   the existence of such a structure is proved only 
  for finite unions of  $\Q$-rational 
    polyhedra of regular at infinity 
  in $\Trop (\P^n)$ (after changing the lattice from $N $ to $\frac{1}{k}N$ ($k \in \Z_{\geq 1}$), which corresponds to a base change of $\mathcal{D}^{*}$). 
   This is why 
   isomorphism (\ref{eq tro coh = LMHS}) in 
  \cite{ItenbergKatzarkovMikhalkinZharkovTropicalhomology2019}
  and its generalization 
  in \cite{AksnesAminiPiquerezShawCOHOMOLOGICALLYTROPICALVARIETIES2025} 
  are considered only for families of algebraic varieties 
  in $\P^n$
   while tropicalizations of algebraic varieties are 
   defined for embeddings to any toric varieties.

  We shall show the existence of such structures after change of the lattice in smooth proper tropical toric varieties, which was conjectured in  \cite{AminiPiquerezHodgetheoryfortropicalvarieties2020} 
  (in a slightly different form). 
  (A weak form (i.e., after a blow-up of the ambient toric variety) was also proved in \cite{AminiPiquerezHodgetheoryfortropicalvarieties2020}.)
  Thus we 
   justify the assumption in \Cref{introduction spectral sequence for ss reduction}, 
   and give a generalization of isomorphism (\ref{eq tro coh = LMHS})  \cite[Theorem 1]{ItenbergKatzarkovMikhalkinZharkovTropicalhomology2019}, 
   and its generalization 
   \cite[Theorem 7.1]{AksnesAminiPiquerezShawCOHOMOLOGICALLYTROPICALVARIETIES2025}.
  (See \Cref{cohomology comparison in sm projective toric varieties:label}.)

  \begin{prp}[\Cref{the existence of unimodular triangulation}, \cite{ItenbergKatzarkovMikhalkinZharkovTropicalhomology2019}
  when $T_{\Sigma} = \P^n$, 
   a weaker form \cite{AminiPiquerezHodgetheoryfortropicalvarieties2020}]
    \label{intro the existence of unimodular triangulation}
  Let $X$ be a finite union of $\Q$-rational polyhedra  regular at infinity 
  in a smooth proper tropical toric variety $\Trop (T_{\Sigma})$
   with the cocharacter lattice $N $.
   Then there is a
   $\frac{1}{k}N$-unimodular
    $\Q$-rational polyhedral complex strucuture 
    of $X$ 
   for some $k \in \Z_{\geq 1}$. 
  \end{prp}

  As a byproduct of our construction of the spectral sequence 
  (\Cref{introduction spectral sequence for ss reduction}), 
  we shall give 
   a new description of ``monodromy'' of tropical varieties, 
   the \emph{eigenwave action}.
  The eigenwave action was introduced by 
  Mikhalkin-Zharkov \cite[Subsection 5.2]{MikhalkinZharkovTropicaleignewaveandintermediatejacobians2014}.
     Its dual 
   $H^{p,q}_{\Trop} (X ) \to H^{p-1,q+1}_{\Trop}(X)$
    coincides (\cite[Appendix]{MikhalkinZharkovTropicaleignewaveandintermediatejacobians2014})
     with the monodromy action $N$ on
     $ \Gr_{2p}^W H^{p+q} (Z_{\infty})$
     in the smooth realizable case.
   The eigenwave action is a cap product, 
    and the idea that the monodromy can be represented by some cap product  
    was appeared in the Calabi-Yau case 
    in works including Gross \cite{GrossSpecialLagrangianfibrations.I.Topology1998}, 
    Zharkov \cite{ZharkovTorusfibrationsofCalabi-Yauhypersurfacesintoricvarieties2000}, and Gross-Siebert \cite{GrossSiebertMirrorsymmetryvialogarithmicdegenerationdataII2010}.
    In particular, the cap product in \cite{GrossSiebertMirrorsymmetryvialogarithmicdegenerationdataII2010}, 
    called the \emph{radiance obstruction}, 
    coincides with the eigenwave action for Calabi-Yau complete intersection of Batyrev-Borisov 
    by a work of 
    Yamamoto \cite{YamamotoTROPICALCONTRACTIONSTOINTEGRALAFFINEMANIFOLDSWITHSINGULARITIES2021}. 
    There is also another description using superforms \cite{JellTROPICALCOHOMOLOGYWITHINTEGRALCOEFFICIENTSFORANALYTICSPACES2022}, \cite{LiuMonodromymapfortropicalDolbeaultcohomology2019}. 
 
  The eigenwave action is 
   not a direct analog of 
   the complex-geometric monodromy action.
  Recall that 
  the monodromy action $N$ 
  on weight graded quotients $\Gr_{2p}^W H^{p+q} (Z_{\infty})$
  can be written by using 
   the \emph{Gauss-Manin connection} 
   of the relative de Rham cohomology sheaves,  
   constructed as the  
   connecting homomorphism of a short exact sequence 
   (Katz-Oda \cite{KatzOdaOnthedifferentiationofDeRhamcohomologyclasseswithrespecttoparameters1968}, Deligne-Katz \cite{KatzTHEREGULARITYTHEOREMINALGEBRAICGEOMETRY1970})
   $$  0 \to \Omega_{\mathcal{D}}^1 (\log 0)  \otimes
    R^* f_* \Omega^{*}_{{\overline{\mathcal{Z}} } /\mathcal{D}} (\log Z_0 ) [-1 ]
        \to
         R^* f_*  \Omega^{*}_{\overline{\mathcal{Z}}} (\log Z_0 )  
        \to
         R^* f_*  \Omega^{*}_{{\overline{\mathcal{Z}}} /\mathcal{D}} (\log Z_0 )  \to 0 ,$$
   where $f \colon \overline{\mathcal{Z}} \to \mathcal{D}$ is a strictly semi-stable model over a disk $\mathcal{D}$ with special fiber $Z_0$.

   We shall introduce a tropical Gauss-Manin connection in the completely same way, and show the following.

  \begin{thm}[\Cref{eigenwave is residue of GM connections:label}]\label{intro eigenwave prp:label}
    Let $X$ be a smooth projective tropical variety regular at infinity. 
    Then the dual
   $$H^{p,q}_{\Trop} (X ) \to H^{p-1,q+1}_{\Trop}(X)$$
    of
   the eigenwave action 
   coincides with the residue of tropical Gauss-Manin connection.
   \end{thm}

  Proof is as follows. 
  Steenbrink \cite{SteenbrinkLimitsofHodgestructures1976} proved that 
  the monodromy action $N$, 
      given by the Gauss-Manin connection, 
     can be written as natural maps of 
    his spectral sequence. 
    Mikhalkin-Zharkov \cite[Proof of Theorem 3]{MikhalkinZharkovTropicaleignewaveandintermediatejacobians2014}
    (in the realizable case) 
    and 
  Amini-Piquerez \cite[Corollary 6.21]{AminiPiquerezHodgetheoryfortropicalvarieties2020} 
    (in general)
   proved that for a smooth projective tropical varieties, 
   the eigenwave action can be written as similar maps  of their spectral sequence.
  Therefore \Cref{intro eigenwave prp:label} follows from a tropical analog of Steenbrink's result,
   which can be proved in the same way.

  The contents of the paper is as follows. 
   In \Cref{sec tropical varieties}, we recall tropical varieties, 
   and prove the existence of $\frac{1}{k}N$-unimodular polyhedral complex structures of 
   tropical varieties regular at infinity. 
   In \Cref{sec:tropical_cohomology}, 
   we recall tropical cohomology, 
   Amini-Piquerez's complex computing tropical cohomology 
   for natural compactifications of tropical fans, 
   and Gysin maps. 
   \Cref{sec:log_tropical_cohomology_and_spectral_sequences} 
   is the main body of this paper. 
   In Subsection \ref{subsec log tropical cohomology}, 
   we introduce log tropical holomorphic forms 
    for tropical fans, 
    and give 
    a generalization of Amini-Piquerez's analog  
  (\cite{AminiPiquerezHomologyoftropicalfans2021},  \cite{AminiPiquerezHomologicalsmoothnessandDeligneresolutionfortropicalfans2024}) of 
    Deligne's weight spectral sequences for smooth open complex varieties (\cite{DeligneThoriedeHodgeII1972})
  from the case of smooth tropical fans
  to general tropical fans.
   In Subsection \ref{sub:semi_stable_reductions_of_tropical_varieties}, 
    we introduce strictly semi-stable reductions of tropical varieties regular at infinity. 
   In Subsection \ref{sub:monodromy_weight_spectral_sequences}, 
    we prove \Cref{introduction spectral sequence for ss reduction} in the completely same was as Steenbrink. 
   In Subsection \ref{sub:eigenwaves}, 
    we introduce tropical Gauss-Manin connections,
     and 
    prove \Cref{intro eigenwave prp:label}.     
   In Appendix, we fix our sign convention.

\subsection*{Acknowledgements}
The main idea of this paper first arose while I was preparing for an invited talk on \cite{AminiPiquerezHodgetheoryfortropicalvarieties2020} 
at the University of Osaka. 
I would like to express my sincere gratitude to 
Yusuke Nakamura and Masahiko Yoshinaga for the kind invitation.
This paper was written while I was a postdoctoral fellow under the supervision of Yuan-Pin Lee
and later under the supervision of Adeel Ahmad Khan.
I would like to express my sincere gratitude to them for their thoughtful support.
  I am currently supported by the National Science and Technology Council under grant number NSTC 112-2628-M-001-006.
I also gratefully acknowledge the support of Academia Sinica and the National Science and Technology Council.

\section{Tropical varieties}\label{sec tropical varieties}
  In this section,
    we fix notations on polyhedral complexes
     (\Cref{subsec:polyhedral complexes}), 
    and 
    recall tropical varieties (\Cref{sub:tropical_varieties}). 
  We also show 
   the existence of 
    quasi-projective $\frac{1}{k}N$-unimodular polyhedral complex structures of  
    a finite union of 
    $\Q$-rational polyhedra regular at infinity (\Cref{subsec:the_existence_of_unimodular_triangulations}). 
  These polyhedral complex structures are used to give 
  strictly semi-stable reductions of projective tropical varieties 
  regular at infinity (\Cref{sub:semi_stable_reductions_of_tropical_varieties}).

 \subsection{Polyhedral complexes} \label{subsec:polyhedral complexes}
   In this subsection,  we fix notations on polyhedral complexes.
  Let
  $M $ be a free $\Z$-module of finite rank $n$,
  and $N:=\Hom(M,\Z)$.
   For a $\Z$-module $G$ and  a commutative ring $R$, we put $G_R := G \otimes_\Z R$.

    \begin{dfn}
    A subset of $\R^n$ is called a \emph{$\Q$-rational polyhedron}
    if  it is the intersection of finitely many subsets of the form
    $$\{x \in \R^n \mid \langle x , a \rangle \leq b \} \ (a \in \Z^n, b \in  \Q),$$
    where $\langle x,a\rangle$ is the usual inner product of $\R^n$.
    A $\Q$-rational polyhedron is said to be 
    \emph{strongly convex} 
    if it does not contain any affine line.
    \end{dfn}

    In this paper, all $\Q$-rational polyhedra 
      are assumed to be strongly convex. 
     A $\Q$-rational polyhedron is called a \emph{cone} when we can take $b=0$ in the definition. 

   We shall use toric geometry, see 
   \cite{CoxaLittleSchenckToricvarieties2011} for details.
  In this paper, fans in $\R^s$ are finite collections of cones satisfying the usual properties.
  Let $\Sigma$ be a fan in $N_\R$, and $T_{\Sigma}$ the toric variety over a field $K$ corresponding to $\Sigma$. 
  For a cone $\sigma \in \Sigma$,
  we put 
  \begin{align*}
  \sigma^{\perp}:= & \{m\in M_\R \mid n(m)=0 \ (n\in \sigma) \}, \\
  \sigma^{\vee}:= &\{m\in M_\R \mid n(m ) \geq 0 \ (n\in \sigma) \}, 
  \end{align*}
  and 
  $N_{\sigma}:= \Hom (M  \cap \sigma^{\perp} , \Z)$.
  We put $T_{\sigma}:= \Spec K[M \cap \sigma^{\vee}]$ the affine toric variety corresponding to  
  $\sigma\in\Sigma$.
  
  We put $\Trop(T_{\Sigma}):=\bigsqcup_{\sigma\in\Sigma}N_{\sigma,\R}$,  endowed with a topology whose open subsets are 
  generated by 
  open subsets of 
  $$\Trop(T_{\Sigma}) \supset 
    \Trop(T_{\sigma}):=
   \bigsqcup_{\substack{\tau\in\Sigma\\
   \tau\subset\sigma } }N_{\tau,\R}
    \cong
   \Hom(M\cap\sigma^{\vee},\R\cup\{\infty\}),
  $$
  where 
  $ \Hom(M\cap\sigma^{\vee},\R\cup\{\infty\})$ is endowed with a natural topology.
  A tropical toric variety $\Trop(T_{\Sigma})$ is called a Kajiwara-Payne partial compactification of $N_{\R}\cong\R^{n}$. 
  See \cite{PayneAnalytificationisthelimitofalltropicalizations2009},
  \cite{RabinoffTropicalanalyticgeometryNewtonpolygonsandtropicalintersections2012} for details.

  In this paper, we will mainly consider a smooth toric variety $T_{\Sigma}$, 
  in this case,  we have $T_{\sigma}\cong\A^{\dim\sigma}\times\G_{m}^{n-\dim\sigma}$,
  and 
  hence 
  \[
  \Trop(T_{\sigma})\cong(\R\cup\{\infty\})^{\dim\sigma}\times\R^{n-\dim\sigma}.
  \]
  
    We shall recall fans and polyhedral complexes in $\Trop (T_{\Sigma})$.

  \begin{dfn}
  A  \emph{$\Q$-rational polyhedron} 
  (resp.\ a \emph{cone}) 
  in $\Trop(T_{\Sigma})$ is the closure $\overline{C}$
  in $\Trop(T_{\Sigma})$ of a $\Q$-rational polyhedron (resp.\ a cone) $C\subset N_{\sigma,\R}$
  for some $\sigma\in\Sigma$. We put $\dim(\overline{C}):=\dim(C)$ and 
  $\relint\overline{C}:=\relint C$. 
  \end{dfn}
  
  Let $P\subset \Trop(T_{\Sigma})$ be a $\Q$-rational polyhedron. 
  We put $\sigma_{P}\in\Sigma$
  the unique cone 
  such that $P\cap N_{\sigma_{P},\R}\subset P$ is dense.
  A subset $Q \subset P$ is called 
  a \emph{face} of $P$ 
  if  $ Q = \overline{R} \cap \overline{N_{\tau,\R}}$
  for a face $R$ of $P \cap N_{\sigma_P, \R}$ (in the usual sense) and a cone $\tau \in \Sigma$ containing $\sigma$
  (where the closures are taken in $\Trop (T_{\Sigma})$).
  A finite collection $\Lambda$ of  $\Q$-rational polyhedra (resp.\ cones) in $\Trop(T_{\Sigma})$
   is called  a   \emph{$\Q$-rational polyhedral complex} (resp.\ a \emph{fan}) 
   if it satisfies the usual properties, i.e., 
  \begin{itemize}
  \item for $P \in \Lambda$, each face of $P$ is also in $\Lambda$, and 
  \item for $P,Q \in \Lambda$, the intersection $P \cap Q$ is a face of $P$ and $Q$.
  \end {itemize}

  We put $\Lambda_i \subset \Lambda$ the subset of $i$-dimensional $\Q$-rational polyhedra.
  We call the union
  $$\lvert \Lambda \rvert := 
  \bigcup_{P\in\Lambda} P \subset \Trop(T_{\Sigma}) $$
  the \emph{support} of $\Lambda$.
  We say that $\Lambda$ is a \emph{$\Q$-rational polyhedral complex} (resp.\ a \emph{fan}) \emph{structure} of $\lvert \Lambda \rvert $.
  A $\Q$-rational polyhedral complex $\Lambda'$ is called a \emph{subdivision} of a $\Q$-rational polyhedral complex $\Lambda$ 
  if their supports are the same and 
  for any $P \in \Lambda'$,
  there exists  $Q \in \Lambda$
  such that $P  \subset  Q.$

 \subsection{Tropical varieties}\label{sub:tropical_varieties}
  In this subsection, we shall recall tropical varieties. 
   For details, see e.g., \cite{ItenbergKatzarkovMikhalkinZharkovTropicalhomology2019}, 
   \cite{MikhalkinZharkovTropicaleignewaveandintermediatejacobians2014}, 
   \cite{JellShawSmackaSuperformstropicalcohomologyandPoincarduality2019}, 
    \cite{MaclaganSturmfelsIntroductiontotropicalgeometry2015}.  

  For a $\Q$-rational polyhedron $P\subset \Trop (T_{\Sigma })$ 
    and
  $R=\Z,\Q$, or $\R$, 
  we put 
  $$\Tan_R P := \Tan_R (P \cap N_{\sigma_P,R})
     := \sum_{ \substack{ x,y\in P\cap  N_{\sigma_P, \R }
              \\
              x-y \in N_{\sigma_P, R} 
          } } R ( x-y ) \subset N_{\sigma_P,R}, $$
  an $R$-submodule of $N_{\sigma_P,R}$.
  (Recall that $\sigma_P \in \Sigma$ is the cone  such that 
  $P \cap N_{\sigma_P,\R} \subset P$ is dense.)

  \begin{dfn}
  \label{def:def of tropical variety}
  In this paper, 
  a ($\Q$-rational) \emph{tropcial variety} 
  $(X, w )$
   of dimension $d$
  in $\Trop(T_{\Sigma})$ 
   is a pair of 
   the support 
    $X \subset \Trop (T_{\Sigma})$
    of 
  a $\Q$-rational polyhedral complex $\Lambda$ of pure-dimension $d$ 
  and 
  a function $w \colon X \to \Z_{\geq 1}$
  such that 
  $w |_{\relint P}$ is constant $w_P$ for $P \in \Lambda_d$,
  and 
  for any $Q\in\Lambda_{d-1}$ of dimension $(d-1)$, we have 
  $$
  \sum_{\substack{P\in\Lambda_{d}\\
  Q\subset P,\sigma_{Q}=\sigma_{P}
  }
  }w_{P}v_{P,Q}=0 
  \qquad  \text{in} \ N_{\sigma_{Q}}/\Tan_{\Z}Q,
   $$ 
   where $v_{P,Q}\in N_{\sigma_{Q}}/\Tan_{\Z}Q$
  is the primitive vector such that 
  $$\R_{\geq0}\cdot v_{P,Q}
   =  \sum_{ \substack{x \in Q \cap N_{\sigma_Q, \R} \\ 
           y \in P \cap N_{\sigma_Q,\R} }} 
           \R_{\geq 0} \cdot \overline{y-x} 
    \qquad \text{in} \ N_{\sigma_{Q},\R}/\Tan_{\R}Q.
   $$ 

   Note that 
   the $\Q$-rational polyhedral structure $\Lambda$ 
   is \emph{not} fixed in the definition.
   In particular, any subdivision of $\Lambda$ can be used.
  \end{dfn}
   For simplicity, by abuse of notation, let $X$ also denote $(X,w)$.
   When we can take a fan $\Lambda$, we call $(X,w)$ a \emph{tropical fan}.

  \begin{dfn}
  \label{def:regular at infinity} 
  We assume that $T_{\Sigma}$ is smooth.
  We say that a subset $P\subset\Trop(T_{\Sigma})$
  is \emph{regular at infinity} (see e.g., \cite[Definition 1.2]{MikhalkinZharkovTropicaleignewaveandintermediatejacobians2014}) 
  if 
   for any $\sigma \in\Sigma$ and $x\in P\cap N_{\sigma ,\R}$, there exists
  an open neighborhood $U_{x}\subset\Trop(T_{\Sigma})$ of $x$ such
  that 
  $$P\cap U_{x}=
  ( (\R \cup \{ \infty \})^{\dim \sigma  } \times (P\cap N_{\sigma ,\R})   )
  \cap U_{x}
  $$ 
   under an identification 
   $$
  \Trop(T_{\sigma})\cong(\R\cup\{\infty\})^{\dim\sigma}
   \times N_{\sigma, \R} .
  $$
  \end{dfn}

  We say that a tropical variety $(X,w)$ is \emph{regular at infinity} if $X$
   is regular at infinity. 
  We say that a $\Q$-rational polyhedral complex is regular at infinity 
   if every maximal dimensional polyhedron is regular at infinity.

  \begin{exam}
  $\Trop (T_\Sigma)$ equipped with constant function $w \colon \Trop (T_\Sigma) \to \{ 1 \}$
  is a tropical variety of dimension $n$. 
  By abuse of notation, we also call this a \emph{tropical toric variety,}
  and let $\Trop(T_{\Sigma})$ denote it. 
  For example, we have $\Trop(\A^{s})=(\R\cup\left\{ \infty\right\} )^{s}$
  for a fixed toric strcture of $\A^{s}$. 
  \end{exam}

 \subsection{The existence of unimodular polyhedral complex structures}
  \label{subsec:the_existence_of_unimodular_triangulations}

  In this subsection, 
    we shall show the existence (\Cref{the existence of unimodular triangulation}) of 
    $\frac{1}{k}N$-unimodular polyhedral complex structures of  
    finite unions of 
    $\Q$-rational polyhedra regular at infinity,
    conjectured by  Amini-Piquerez
    (\cite[Subsection 1.12]{AminiPiquerezHodgetheoryfortropicalvarieties2020}) 
    (in a slightly different form). 
  These polyhedral complex structures are used to give 
  strictly semi-stable reductions of tropical varieties 
  in \Cref{sub:semi_stable_reductions_of_tropical_varieties}.
   The existence was previously proved by Itenberg-Katzarkov-Mikhalkin-Zharkov 
    (\cite[Proposition 48]{ItenbergKatzarkovMikhalkinZharkovTropicalhomology2019})
    when the ambient toric variety is a projective space.
   A weaker form 
    (i.e., up to blow-ups of the ambient toric variety) was proved by Amini-Piquerez
    (\cite[Theorem 4.1]{AminiPiquerezHodgetheoryfortropicalvarieties2020}).

   Let $\Sigma$ be a complete unimodular fan in $N_\R$.

  \begin{dfn}\label{def unimodularity:label}
  A $\Q$-rational polyhedron $P$ 
   in $\Trop (T_{\Sigma})$
   is said to be \emph{$\frac{1}{k}N$-unimodular} for  $k \in \Z_{\geq 1}$ 
    if 
    there exist 
     $v_i \in \frac{1}{k}N $ ($ 0 \leq i \leq r$)
     and $u_j \in N $ ($ 1 \leq j \leq \dim P -r$)
    such that 
    \begin{itemize}
      \item
     $
        \R_{\geq 0}  \langle  u_j \rangle_{ 1 \leq j \leq \dim P - r } $
        is contained in  some unimodular projective fan $\overline{\Sigma} \supset \Sigma$,
     \item
     $ v_i - v_0, \frac{1}{k}u_j $ ($1 \leq i \leq r$, $1 \leq j \leq \dim P -r $)
     extend to a basis of $\frac{1}{k}N$,
     and 
     \item 
    $P$ is the closure of the Minkowski sum
    $$  \Conv (v_0, \dots, v_r) 
    + \R  \langle  u_j \rangle_{ 1 \leq j \leq \dim P - r } $$
    where 
      $\Conv $ is the convex hull, 
    \end{itemize}
    where 
      $\frac{1}{k}N$ is the scalar multiplication of $N$.
  A $\Q$-rational polyhedral complex is 
  said to be $\frac{1}{k}N$-unimodular 
  if every maximal polyhedron is $\frac{1}{k}N$-unimodular.
  \end{dfn}

  Let $X\subset \Trop (T_\Sigma)$ be a finite union of $\Q$-rational polyhedra regular at infinity. 

  \begin{prp}\label{the existence of unimodular triangulation}
   There is a
   $\frac{1}{k}N$-unimodular
    $\Q$-rational polyhedral complex strucuture $\Lambda_X$ 
    of $X$ 
   for some $k \in \Z_{\geq 1}$
   such that 
   $\Lambda_X |_{ N_\R} := \{ P \cap N_\R \}_{ P \in \Lambda_X }$
     is finer than $\Sigma$, 
     i.e., 
     for any $ P \in \Lambda_X$, 
     there exists a cone $\sigma \in \Sigma$ such that 
     $P  \cap N_\R \subset \sigma$.
  \end{prp}

   When $\Sigma$ is projective, 
   $\Lambda_X$ is ``quasi-projective'' in a sense, see \Cref{quasi-projectivity of polyhedral complex structures} for details. 

  \begin{rem}
   We can also
   apply  
   \Cref{the existence of unimodular triangulation} 
   for quasi-projective $\Sigma$.
   In this case, $\Lambda_X |_{N_\R}$ is finer than 
   a unimodular projective fan $\overline{\Sigma}$ containing $\Sigma$
   such that the closure of $X$ in $\Trop (T_{\overline{\Sigma}})$ is again regular at infinity.
   The existence of such a $\overline{\Sigma}$
   follows from 
   toric resolution of singularities (\cite[Theorem 11.1.9]{CoxaLittleSchenckToricvarieties2011}). 
  \end{rem}

  The idea of proof of \Cref{the existence of unimodular triangulation} is 
  extending  
   a polyhedral complex structure 
   of a compact subset  of 
   $X \cap N_\R$,
   given by the following theorem, 
     by taking Minkowski sums with elements in $ \Sigma$.

  \begin{thm}[{Kempf-Knudsen-Mumford-Saint-Donat
  \cite[Chapter III, Theorem 4.1]{KempfKnudsenMumfordSaint-DonatToroidalEmbeddings11973}}]
    \label{KKMSD unimodular triangulation:label}
    For a compact $\Q$-rational polyhedral complex 
      in $N_\R$,
    there exists a $\frac{1}{k} N $-unimodular subdivision  for some $k \in \Z_{\geq 1}$.
  \end{thm}

  The compact subset of $X \cap N_\R$ to which we apply 
   \Cref{KKMSD unimodular triangulation:label} 
   is of the form $X \cap rQ_{ \{ 0 \} }$ ($ r \gg 0$), defined below.
  For each $1$-dimensional cone $l \in \Sigma_1$, 
   we put $v_l \in l \cap N$ the primitive vector. 
    Let $ r >0$ be a rational number.
    We put
    $$ rQ_{ \{ 0 \} } 
       := \bigcup_{ \sigma \in \Sigma} 
       \{ \sum_{ \sigma \supset l \in \Sigma_1 } 
           a_l v_l \mid 
               \sum_{ \sigma \supset l \in \Sigma_1 } 
                  a_l \leq  r,  \ 0 \leq a_l \leq r\} , $$
      an $n$-dimensional compact polyhedron, 
     and for a cone $\sigma \in \Sigma \setminus \{ \{ 0 \} \}$, 
    we also put 
    $$ rQ_{\sigma} := \{ \sum_{ \sigma \supset l \in \Sigma_1 } 
           a_l v_l \mid 
               \sum_{ \sigma \supset l \in \Sigma_1 } 
                  a_l = r, \ 0 \leq a_l \leq r  \} = \partial (rQ_{\{ 0 \}}) \cap \sigma. $$

    The polyhedral complex structure $ \{r Q_{\sigma}\}_{\sigma \in \Sigma }$ of $r Q_{ \{ 0 \} }$ 
    gives a stratification 
    $$N_\R = \bigsqcup_{\sigma \in \Sigma} \relint (r Q_{\sigma}) + \sigma,$$
    where the sums are the Minkowski sums. 
    A naive idea of the extension of a  polyhedral complex structure of 
    $X \cap r Q_{ \{ 0 \} }$ is 
    taking $ U + \sigma$ for $U \subset rQ_{\sigma}$ and $\sigma \in \Sigma$.
    However, 
    this does not work well. 
    Roughly speaking, this is 
    because $r Q_{\sigma} + \sigma$ is not full-dimensional.
    We thicken the lower-dimensional starata of 
    this stratification
    using 
     the barycentric subdivision $ r \hat{Q}$ of 
    the polyhedral complex structure $ \{r Q_{\sigma}\}_{\sigma \in \Sigma }$ of $r Q_{ \{ 0 \} }$. 
    The subdivision $ r \hat{Q}$ can be described explicitly as follows.
    There is a bijection between cones of $\Sigma$ and vertices of $r \hat{Q}$ given by 
     $$ \{ 0 \} \neq \sigma \mapsto \sum_{ \sigma \supset l \in \Sigma_1 } \frac{r}{\dim \sigma} v_l $$
     and $\Sigma \ni \{0 \} \mapsto 0 \in N_\R$.
     More generally, 
     there is a bijection 
     between flags $\sigma_{*}:=( \sigma_1 \subsetneq \dots  \subsetneq \sigma_s )$ in 
     $\Sigma$
     and polyhedra 
     $$P(\sigma_*) :=   \Conv \bigg\{ \sum_{ \sigma_i \supset l \in \Sigma_1 } \frac{r}{\dim \sigma_i} v_l 
     \bigg\}_{1\leq i\leq s} $$ in $r \hat{Q}$, 
     where $\Conv$ is the convex hull.
    For  $\sigma \in \Sigma$, 
     we put 
       $$ \overline{rP_{\sigma}} 
            := \bigcup_{ 
            \substack{ \sigma_* =(\sigma_1 \subsetneq \dots  ) \\ \sigma_1=\sigma } }
             P(\sigma_*), 
             \quad 
              rP_{\sigma}  := 
          \overline{rP_{\sigma}} 
           \setminus  \bigcup_{\sigma \subset \sigma' \in \Sigma}
          \overline{rP_{\sigma' }} .
          $$
          Then $rP_{\sigma}$
          is a neighborhood of 
          $\sum_{ \sigma \supset l \in \Sigma_1 } \frac{r}{\dim \sigma} v_l$
     in $\partial (rQ_{ \{0\} })$, 
     and 
      $
      \bigsqcup_{\sigma \supset \tau \in \Sigma \setminus \{ \{ 0 \} \}}  rP_{\tau}  $
      is a neighborhood of 
     $ rQ_{\sigma } $
     in $\partial (rQ_{ \{0\} })$.
      We have  
    stratifications 
      $$ rQ_{ \{ 0 \} } 
          = \bigsqcup_{\sigma \in \Sigma} rP_{\sigma} , 
          \quad 
        N_\R = 
           \bigsqcup_{\sigma \in \Sigma} rP_{\sigma} + \sigma .$$

    We put 
    $$\Psi_\tau \colon N_\R= \Hom (M,\R) \to 
      \Hom (M\cap \tau^{\perp} , \R) = N_{\tau,\R} \cong N_\R / \R \tau$$
    the natural map. 
    Then 
     $\Psi_{\tau}|_{rP_{\tau}} \colon rP_{\tau} \to N_{\tau,\R} $ is injective, 
     and 
     preserves lattice structures 
   because $\Sigma$ is unimodular,
     and 
   we have 
   $ \bigcup_{r >0} \Psi_{\tau} (rP_{\tau}) = N_{\tau,\R}$
    ($\tau \in \Sigma$) 
    (in particular, for any compact subset $C \subset N_{\tau, \R}$, we have  
    $  C \subset  \overline{rP_{\tau}+ \tau} $ for sufficently large $r >0 $).
  These properties are what we need to extend a  
   $\frac{1}{k}N$-unimodular
    $\Q$-rational polyhedral complex structure of 
    $X \cap r Q_{ \{ 0 \} }$ to that of $X$.

  \begin{lem}\label{lem for unimodular triangulation:label}
   Let $\sigma \in \Sigma \setminus \{ \{ 0 \} \} $ be a cone.
    Then we have 
    \begin{align*}
      rQ_{\sigma } +\sigma 
    \subset &
      \bigsqcup_{\sigma \supset \tau \in \Sigma \setminus \{ \{ 0 \} \}}  rP_{\tau} + \tau  , \\
     \sigma 
    \subset &
      \bigsqcup_{\sigma \supset \tau \in \Sigma }  rP_{\tau} + \tau .
    \end{align*}
  Moreover, 
   there exists $\epsilon >0$ such that 
  $$[-r\epsilon, r\epsilon]^n + \sigma 
    \subset 
      \bigsqcup_{\sigma \supset \tau \in \Sigma }  rP_{\tau} + \tau.
      $$
  \end{lem}
  \begin{proof}
  The assertions easily follow from the construction.
    (Note that $rP_{ \{ 0 \} } = \mathrm{int}  (rQ_{ \{ 0 \} })$.)
    \end{proof}

  \begin{exam}\label{exam for unimodular triangulation:label}
   Let $\sigma \in \Sigma$ be of the form
     $$\sigma = \R_{\geq 0}^2 \times \{ 0 \}^{n-2}  \subset \R^n = N_\R, $$
    and $ \sigma \supset l_1,l_2 \in \Sigma$ 
     $$l_1 = \R_{\geq 0} \times \{ 0 \}^{n-1}    , \quad 
     l_2 = \{ 0 \} \times \R_{\geq 0} \times \{ 0 \}^{n-1}  . $$
    Then we have 
    \begin{align*}
     rQ_{\sigma } 
    = & \{  (x_i)_{ i  } \in \R^n 
     \mid 
                  x_1+x_2 = r  , \
                   x_i \geq 0 \ (i=1,2), \ 
                    x_j =0 \ (j \geq 3)
                    \} ,  \\
     rQ_{\sigma } \cap rP_{l_a} 
    =  &\{  (x_i)_{ i  } \in rQ_{\sigma } 
     \mid 
                    \frac{1}{2} r <  x_a \leq r 
                    \}  
                    \quad (a=1,2) ,  \\
     rQ_{\sigma } \cap rP_{\sigma} 
    =&  \{ (\frac{1}{2} r, \frac{1}{2} r, 0, \dots, 0 )
                    \}    .
      \end{align*}
      Hence we have $$
     rQ_{\sigma } 
    \subset  
       rP_{l_1} \sqcup
       rP_{l_2} \sqcup
       rP_{\sigma}
        .$$
    We also have 
    \begin{align*}
     rQ_{\sigma } + \sigma 
    = & \{  (x_i)_{ i  } \in \R^n 
     \mid 
                  x_1+x_2 \geq  r  , \
                   x_i \geq 0 \ (i=1,2), \ 
                    x_j =0 \ (j \geq 3)
                    \} ,  \\
     (rQ_{\sigma } + \sigma) \cap ( rP_{l_{a}} +l_{a})
    =  &\{  (x_i)_{ i  } \in rQ_{\sigma } + \sigma 
     \mid 
                    0 \leq   x_{b}  < \frac{1}{2} r 
                    \}  \quad  (\{a,b \}=\{1,2 \}) ,  \\
     (rQ_{\sigma } + \sigma) \cap (  rP_{\sigma} + \sigma)
    =&  \{  (x_i)_{ i  } \in rQ_{\sigma } + \sigma 
     \mid 
                  x_i \geq \frac{1}{2} r \ (i=1,2) 
                    \}    .
      \end{align*}
      Hence we have $$
     rQ_{\sigma } + \sigma
    \subset  
       (rP_{l_1} + l_1 ) \sqcup
       (rP_{l_2} + l_2  )\sqcup
       (rP_{\sigma} + \sigma )
        .$$
    We also have 
    $$ rP_{ \{ 0 \} } \supset 
        \{  (x_i)_{ i  } \in \R^n 
     \mid 
                  x_1+x_2 < r  , \
                   x_i \geq 0 \ (i=1,2), \ 
                    x_j =0 \ (j \geq 3)
                    \} .$$
      Hence we have $$
      \sigma
    \subset  
       rP_{ \{ 0 \} } \sqcup 
      ( rP_{l_1} + l_1  )\sqcup
      ( rP_{l_2} + l_2  )\sqcup
      ( rP_{\sigma} + \sigma )
        .$$
  \end{exam}

  \begin{lem}\label{existence of nice dual polytope}
   For a sufficiently large integer $r>0$, 
   we have 
      $$ X \cap N_\R = \bigsqcup_{
         \tau \in \Sigma   } 
      ( X \cap r P_{\tau })  + \tau. $$
    \end{lem}  
   \begin{proof}
     Let $Q \subset N_\R$ be a 
     compact polyhedron, 
        and  $\sigma \in \Sigma$.
    By \Cref{lem for unimodular triangulation:label}, 
    for a sufficiently large $r >0$, 
     we have  
     $$ Q + \sigma \subset 
      \bigsqcup_{\sigma \supset \tau \in \Sigma }  rP_{\tau} + \tau, $$
    i.e., we have 
     $$ Q + \sigma = 
      \bigsqcup_{\sigma \supset \tau \in \Sigma } 
      (Q + \sigma) \cap ( rP_{\tau} + \tau). $$
      Since $X$ is a union of subsets of the form $Q + \sigma$, 
    we may assume that 
     $X$ is the closure of 
     a single Minkowski sum $ Q + \sigma $.
     Therefore 
    it suffices to show that 
      $$(Q + \sigma) \cap ( rP_{\tau} + \tau) \subset  
      ( (Q + \sigma) \cap  rP_{\tau} ) + \tau $$
    ($ \sigma \supset \tau \in \Sigma$).

     We fix $ \tau \in \Sigma$ contained in $\sigma$.
     We may assume that 
     \begin{align*}
     \sigma = &  \R_{\geq 0}^{\dim \sigma} \times \{ 0 \}^{n-\dim \sigma}  \subset \R^n = N_\R  \\
     \tau = & \R_{\geq 0}^{\dim \tau} \times \{  0\}^{n-\dim \tau} \subset \sigma .
     \end{align*}
    We put 
    $$\mu := \{ 0 \}^{\dim \tau } \times \R_{\geq 0}^{\dim \sigma - \dim \tau } \times \{ 0 \}^{n-\dim \sigma} \in \Sigma.$$
     Let 
     $$y =q+a_{\tau,1} + b_\mu
         = x + a_{\tau,2} \in (Q + \sigma) \cap ( rP_{\tau} + \tau) $$
     ($q \in Q$, $a_{\tau,i} \in \tau$ ($i =1,2$), $b_\mu \in \mu$, $x \in rP_\tau$) be an element. 
     It is enough to show 
      that for a sufficiently large $r >0$, 
      we have  $x-q-b_\mu = a_{\tau,1}-a_{\tau,2} \in \tau$, i.e., 
      $$y= x +a_{\tau,2} = 
      (q+b_{\mu}+ (a_{\tau,1}-a_{\tau,2}) )+a_{\tau,2} \in ( (Q + \sigma) \cap  rP_{\tau} ) + \tau 
      .$$

     This is not difficult to see, 
     but for convenience, we give a detailed proof.
      Let $R>0$ be such that $Q \subset [-R,R]^n$.
      Then since $x \in q+b_\mu + \R \tau$, we have 
      $$x =(x_i)_{1 \leq i \leq n} \in rP_{\tau} \cap ( \R^{\dim \tau } \times [-R,\infty)^{\dim \sigma - \dim \tau} \times [-R,R]^{n -\dim \sigma}).$$
      There are $1$-dimensional cones $l_j \in \Sigma_1$ ($j \in J$) 
         such that 
         $\tau + \sum_{ j \in J } l_j  $ 
         is a cone in $\Sigma$ of dimension $(\dim \tau +s)$,
         and 
      $$x = \sum_{ \tau \supset l \in \Sigma_1 } 
        a_l v_l + \sum_{ j \in J } a_j v_{l_j}$$
        for some $a_l, a_j >0 $ 
        (with 
      $ \sum_{ \tau \supset l \in \Sigma_1 } 
        a_l  + \sum_{ j \in J } a_j =r$).
     We have   
      $$\Psi_{\tau} (x) = \sum_{  j \in J } a_j \Psi_{\tau} (v_{l_j}) \in 
      [-R,\infty)^{\dim \sigma - \dim \tau} \times [-R,R]^{n -\dim \sigma}.$$
     (Note that the natural map 
     $\Psi_{\tau} \colon \R^n = N_\R \to N_{\tau,\R} = \R^{n-\dim \tau }$ is 
     the projection to the last $(n-\dim \tau)$-components.)
      Let $\Star_{\Sigma} \tau 
           := \{ \Psi_{\tau} (\nu)  \}_{\tau \subset \nu \in \Sigma }$ 
            be the star fan in $\R^{n-\dim \tau}$.
      Let $K \subset J$ be the subset consisting of  $j $ with $l_j \not\subset \sigma$.
      Since both
      $$ \sum_{j \in K} 
      \Psi_{\tau}(l_j), 
      \ 
      \Psi_\tau (\sigma) = 
      \R_{\geq 0}^{\dim \sigma - \dim \tau} \times \{0\}^{n -\dim \sigma} $$
      are 
      cones in $ \Star_{\Sigma} \tau$, 
      we have 
      $$\bigg( \sum_{j \in K} 
      \Psi_{\tau}(l_j)  \bigg)
       \cap 
      (\R_{\geq 0}^{\dim \sigma - \dim \tau} \times \{0\}^{n -\dim \sigma} ) 
       = \{ (0, \dots,0) \}. $$
      Hence
      by
      the strictly convexity of $\sum_{ j \in K } 
      \Psi_{\tau}(l_j) $, 
      there exists $A(R)>0$ 
      such that for any 
      $$ \sum_{  j \in J } b_j \Psi_{\tau} (v_{l_j})
       \in 
       \bigg(\sum_{j \in K} 
      \Psi_{\tau}(l_j)  \bigg)
      \cap 
      ([-R,\infty)^{\dim \sigma - \dim \tau} \times [-R,R]^{n -\dim \sigma}),  $$
      we have 
     $b_j < A(R)$ $(j \in K)$.
     In particular, we have 
     $a_j < A(R)$ $(j \in K)$.
       Therefore  
      the $i$-th component of $a_j v_{l_j}$ 
     ($1 \leq i \leq \dim \tau$, $j \in J$)
      is bounded by constant. 
      Hence for sufficiently large $r>0$, 
      we have $x_i >R$ for $1 \leq i \leq \dim \tau$.
      Since 
      $x - q-b_\mu \in \R \tau$, 
      this means that 
      $x - q- b_\mu \in \tau$.
      We have finished the proof.
   \end{proof}

  \begin{proof}{(Proof of Proposition \ref{the existence of unimodular triangulation})}
    Let $r \in \Q_{>0}$ be as in Lemma \ref{existence of nice dual polytope}. 
    By Theorem \ref{KKMSD unimodular triangulation:label},  
    there exist
     $k \in \Z_{\geq 1}$
     and 
    a $\frac{1}{k}N$-unimodular polyhedral complex structure 
    $\Lambda_{ rQ_{ \{ 0 \} }  }$
       finer than 
         the barycentric subdivision $ r \hat{Q}$ of $rQ_{ \{ 0 \} }$
         such that 
      $X \cap rQ_{ \{ 0 \} }$
      is the support of a subcomplex 
      $\Lambda_{X \cap rQ_{ \{ 0 \} } }$ of $\Lambda_{rQ_{ \{ 0 \} }}$.
      Then since the injective map 
      $ \Psi_{\tau}|_{r P_\tau} \colon  r P_\tau \to N_{\tau, \R}$  ($\tau \in \Sigma$) preserves lattice structures, 
   by 
     Lemma \ref{existence of nice dual polytope}, 
      a  polyhedral complex structure 
    $$ 
     \Lambda_X:= 
       \big\{  \overline{U +\tau} \cap \overline{N_{\nu,\R}}  
      \big\}_{
         \substack{ \nu, \tau \in \Sigma , 
               \ \nu \subset \tau 
          \\
             \overline{rP_{\tau  }} 
            \supset U \in \Lambda_{X \cap rQ_{ \{ 0 \} } }
               }}
     $$
     of $X$ is $\frac{1}{k}N$-unimodular by definition, 
     where the closures are taken in $\Trop (T_{\Sigma})$.
     For $U$ and $\tau$ as above, 
     there is a 
      flag $\sigma_{*}= (\sigma_1 \subsetneq \dots  \subsetneq \sigma_s ) $ in 
     $\Sigma$ 
     such that 
     $\sigma_1=\tau$ 
     and $U \subset P(\sigma_*)$.
     Then we have 
     $ R +\tau \subset P(\sigma_*)+\sigma_1 \subset \sigma_s$.
     Hence  $\Lambda_X |_{N_\R}$
     is finer than $\Sigma$.
  \end{proof}

  Proposition \ref{the existence of unimodular triangulation}
  was proved only for $T_{\Sigma} =\P^n$ in 
  \cite{ItenbergKatzarkovMikhalkinZharkovTropicalhomology2019}.
   This is why 
  \cite[Theorem 1]{ItenbergKatzarkovMikhalkinZharkovTropicalhomology2019}
  (isomorphism (\ref{eq tro coh = LMHS}) in \Cref{sec:introduction})
  and its generalization 
  \cite[Theorem 7.1]{AksnesAminiPiquerezShawCOHOMOLOGICALLYTROPICALVARIETIES2025} 
  are considered only for maximally degenerate families 
  in $\P^n$.
  Therefore, 
  as an application 
   of Proposition \ref{the existence of unimodular triangulation}, 
  we can generalize the above results 
  to the case of maximally degenerate families in a smooth proper toric variety $T_{\Sigma}$ .
  Only Proposition \ref{the existence of unimodular triangulation} is new, but for the convenience, we state theorem 
  (\Cref{cohomology comparison in sm projective toric varieties:label}) explicitly.
  
  Let  
   $\mathcal{Z} = (Z_t)_{t \in \mathcal{D}^*} \subset T_{\Sigma} \times \mathcal{D}^{*}$ 
   be a family 
   of smooth complex varieties 
   over the punctured disk $\mathcal{D}^*$
   such that $\Trop(\mathcal{Z}_{\mathbb{Z}((t))})$ is regular at infinity.
   By Proposition \ref{the existence of unimodular triangulation}, 
   there is a 
   $\frac{1}{k}N$-unimodular
    $\Q$-rational polyhedral complex strucuture $\Lambda$
    of $\Trop(\mathcal{Z}_{\mathbb{C}((t))})$.
   After a base change, 
   by proof of semi-stable reduction theorem (\cite{KempfKnudsenMumfordSaint-DonatToroidalEmbeddings11973}), 
   this $\Lambda$ gives 
   a strictly semi-stable reduction 
    $Z_0 = \bigcup_{i \in I} Z_i$ 
   of $\mathcal{Z}$ (see [loc.cit.], 
  \cite{ItenbergKatzarkovMikhalkinZharkovTropicalhomology2019}
  and 
  \cite{AksnesAminiPiquerezShawCOHOMOLOGICALLYTROPICALVARIETIES2025} 
  for details).

   We assume that 
   $$ 
    \bigoplus_{p+q=r} 
    H^{p,q}_{\Trop} \big( \Trop \big( \bigcap_{j \in J} Z_j \big); \Q \big) \cong 
      H^{r}_{\sing} \big( \bigcap_{j \in J} Z_j ; \Q \big) $$
      ($r \geq 0$, $J \subset I$).
  (See \Cref{sec:tropical_cohomology} for tropical cohomology.) 
      This (iso)morphism is defined in \cite[Definition 3.1]{AksnesAminiPiquerezShawCOHOMOLOGICALLYTROPICALVARIETIES2025}.  
      (At least, in this case) this morphism coincides with the one constructed in a more cohomological way in \cite{MikamiDifferentialformsandcohomologyintropicalandcomplexgeometry2021} and \cite[Section 7]{MikamiTropicalIntersectionHomology2024}.
      Under this assumption, 
      $\Gr_{2p}^W H^{p+q} (Z_{\infty})$ is of pure $(p,p)$-type,  and 
   the odd weight quotients vanish. 

  \begin{thm}[\cite{ItenbergKatzarkovMikhalkinZharkovTropicalhomology2019} 
  when $X$ is smooth and $T_{\Sigma} =\P^n$,
   \cite{AksnesAminiPiquerezShawCOHOMOLOGICALLYTROPICALVARIETIES2025} 
   when $T_{\Sigma} =\P^n$] \label{cohomology comparison in sm projective toric varieties:label}
   
   In the above situation, 
   we have 
   \begin{align*}
   H^{p,q}_{\Trop}(
    \Trop(\mathcal{Z}_{\mathbb{C}((t))}))
   \cong 
   \Gr_{2p}^W H^{p+q} (Z_{\infty}).
   \end{align*}
  \end{thm}
  In particular, 
   the hodge numbers of general members $Z_t$
   equal the tropical hodge numbers of $
    \Trop(\mathcal{Z}_{\mathbb{C}((t))})$.
    \begin{proof}
    Proof is the same as 
   \cite[Theorem 1]{ItenbergKatzarkovMikhalkinZharkovTropicalhomology2019}
   and 
   \cite[Theorem 7.1]{AksnesAminiPiquerezShawCOHOMOLOGICALLYTROPICALVARIETIES2025}.
   (See also \Cref{sec:introduction}.)
    \end{proof}

\section{Tropical cohomology}\label{sec:tropical_cohomology}
  In this section, we recall tropical cohomology \cite{ItenbergKatzarkovMikhalkinZharkovTropicalhomology2019}, 
   and Amini-Piquerez's complex \cite{AminiPiquerezHomologyoftropicalfans2021}, \cite{AminiPiquerezTropicalFeichtner-Yuzvinskyandpositivitycriterionforfans2024}, 
   and Gysin maps \cite{AminiPiquerezHodgetheoryfortropicalvarieties2020},
   which will be used to describe     
     the differential $d_1 $ in \Cref{introduction spectral sequence for ss reduction} (\Cref{differential of monodromy-weight spectral sequence:label}, see also \Cref{sign of log spectral sequence:label}). 

   Recall that 
   $M$ is a free $\Z$-module of finite rank and $N:=\Hom(M,\Z)$.

 \subsection{Tropical cohomology}\label{subsec tropical cohomology}
   In this subsection, we briefly recall tropical analog of singular cohomology, \textit{tropical cohomology}, 
    introduced by Itenberg-Katzarkov-Mikhalkin-Zharkov \cite{ItenbergKatzarkovMikhalkinZharkovTropicalhomology2019}.
  See also 
   \cite{MikhalkinZharkovTropicaleignewaveandintermediatejacobians2014},
   \cite{JellShawSmackaSuperformstropicalcohomologyandPoincarduality2019}, 
   \cite{AminiPiquerezHodgetheoryfortropicalvarieties2020}, \cite{GrossShokirehAsheaf-theoreticapproachtotropicalhomology23}.
  Let $T_{\Sigma}$ be the toric variety corresponding to a fan $\Sigma$ in $N_\R$.
  Let $\Lambda $ be a $\Q$-rational polyhedral complex in $\Trop(T_{\Sigma})$.
  Recall that
  for $P \in \Lambda$,
  we put $\sigma_P \in \Sigma$ the cone such that $\relint(P) \subset N_{\sigma_P,\R}$, 
  and $\Tan_\Q P := \Tan_\Q (P \cap N_{\sigma_P,\R})$.
  We put $X:= \lvert \Lambda \rvert$.
  
  Let $p \geq 0$ be a non-negative integer.
  For $P \in \Lambda$,
  we put 
  $$F_p^X(P):= \sum_{\substack{P \subset P' \in \Lambda
     \\ \sigma_{P'} = \sigma_P }}
      \bigwedge^p \Tan_\Q (P') \subset \bigwedge^p N_{\sigma_P,\Q},$$ 
  and  put
  $$F^p_X(P):= \bigwedge^p  (M \cap \sigma_P^{\perp} )_\Q 
  \big/ 
  \big\{f \in \bigwedge^p (M \cap \sigma_P^{\perp})_\Q 
  \big| \alpha(f)=0 \ (\alpha \in F_p^X(P))
  \big\} $$
  the dual of $F_p^X (P)$,
  where we identify $$\bigwedge^p N_{\sigma_P,\Q} \cong \Hom \big(\bigwedge^p  (M \cap \sigma_P^{\perp} )_\Q, \Q \big).$$
  Note that $F_p^X (P)$ and $F^p_X (P)$ does not depend on the $\Q$-rational polyhedral complex structure $\Lambda$ 
     of $X$. 
  When there is no confusion, we simply write $F_p(P)$ (resp.\ $F^p(P)$).
  
  Let $P_1,P_2 \in \Lambda$ with $P_2 \subset P_1$.
  Then we have $\sigma_{P_1} \subset \sigma_{P_2}$.
  The restriction map 
   $$ N_{\sigma_{P_1},\R } 
    = \Hom ( M \cap \sigma_{P_1}^{\perp } , \R ) 
   \twoheadrightarrow N_{\sigma_{P_2},\R }
    = \Hom ( M \cap \sigma_{P_2}^{\perp } , \R ) 
    $$ induces
  a morphism 
  $  F_p(P_1)  
      \to F_p(P_2).$
  Hence we get the dual morphism 
  $ i_{P_2 \subset P_1 } \colon  F^p(P_2)  
      \to F^p(P_1).$
  Consequently, we get a constructible sheaf $\F^p_X$ on $X$
   with respect to $\Lambda$
     such that for  each $P \in \Lambda $, 
        the stalk $\F^p_{X,x}$ is isomorphic to $F^p ( P )$ 
         ($x \in P$) 
         and restrictions of sections are given by 
      $ i_{P_2 \subset P_1 } $ ($P_2 \subset P_1$). 
     (See 
     \cite[Section 2.4]{MikhalkinZharkovTropicaleignewaveandintermediatejacobians2014} or 
     \cite[Section 3.1]{JellShawSmackaSuperformstropicalcohomologyandPoincarduality2019} for details.)   
  We put 
  $$H_{\Trop}^{p,q}(X ) := H^{ q}(X,  \F^p_X), $$
  \emph{tropical cohomology}. 
  For a close subset $j \colon D \hookrightarrow X$, 
  we put 
   $H_{\Trop,D}^{p,q}(X)$ the cohomology group 
   of $j^!\F_X^p$.

  We have the usual wedge product 
    $$ -\wedge - \colon \bigwedge^p M  \otimes \bigwedge^q M   
        \to \bigwedge^{p+q} M  $$
        given by 
    $$ m_1 \wedge \dots \wedge m_p \otimes
       m_{p+1}  \wedge \dots \wedge m_{p+q} 
       \mapsto 
     m_1 \wedge \dots \wedge m_{p+q} .$$ 
  This induces a product 
    $$  (-1)^{pq} \cdot  - \wedge - \colon \F^p_X [-p] \otimes \F^q_X [-q] 
     \ni \alpha \otimes \beta 
     \mapsto 
     (-1)^{pq} \alpha \wedge \beta 
     \in \F^{p+q}_X [-p-q]$$
  and hence a cup product 
    $$ - \cup - \colon H^{p,r}_{\Trop} (X) \times H^{q,s}_{\Trop} (X) 
       \to H^{p+q,r+s}_{\Trop} (X).$$

  For another tropical variety $Y$ contained in another tropical toric variety $\Trop (T_{\Sigma_Y})$, 
   the set-theoretical product $X \times Y \subset \Trop (T_{\Sigma} \times T_{\Sigma_Y})$
   has a natural structure of a tropical variety given by products of weights. 
  We obviously have an isomorphism 
   \begin{align}\label{eq. Kunneth F level}
     ((-1)^{pq} \cdot  - \wedge -)_{p,q} 
    \colon 
     \bigoplus_{ p + q = r } \F^p_X [-p] \boxtimes \F^q_Y [-q] \cong 
    \F^{r}_{X \times Y} [-r] .
    \end{align}
  By the theory of constructible sheaves ([Ach21, Proposition 2.9.1],  which treats the case of algebraic varieties, and whose proof applies equally well in our setting), 
  this induces the K\"{u}nneth formula 
  (\cite[Theorem B]{GrossShokirehAsheaf-theoreticapproachtotropicalhomology23}, cf. \cite{SmackaPhD2017})
  \begin{align} \label{eq. Kunneth formula}
     \bigoplus_{ p + q = r } 
    \bigoplus_{s+ t = u} 
      H^{p,s}_{\Trop} (X) \otimes H^{q,t}_{\Trop} (Y) 
       \cong H^{r,u}_{\Trop} (X \times Y) .
  \end{align}

  \begin{exam}
   Let $\R\cup \{\infty\} = \Trop (\A^1)$ be a tropical toric variety. 
   Then we have an exact sequence 
    $$ 0 = H^{1,0}_{\Trop} (\R \cup \{\infty\}) 
         \to  H^{1,0}_{\Trop} (\R ) 
         \to  H^{1,1}_{\Trop, \{ \infty \} } (\R \cup \{\infty\}) 
         \to  H^{1,1}_{\Trop} (\R \cup \{\infty\}) =0.$$
    Hence the connecting homomorphism 
         $$ \delta_{\R} \colon H^{1,0}_{\Trop} (\R ) 
         \cong  H^{1,1}_{\Trop, \{ \infty \} } (\R \cup \{\infty\}) $$
      is an isomorphism.
    We also have a natural isomorphism  
     $\Q \cdot x \cong  H^{1,0}_{\Trop} (\R ) $, 
     where $ x $ is the coordinate function of $\R \cup \{ \infty \}$, considered as an element of $\F^1 (\R) = F^1 (\R)$.
    Consequently, we have an isomorphism 
     \begin{align} \label{eq. signs H^1,1 R }
        H^{1,1}_{\Trop, \{ \infty \} } (\R \cup \{\infty\}) 
       \cong \Q
     \end{align}
      sending $\delta_{\R} (x)$ to $1$.
  \end{exam}

  \begin{exam}\label{exam: connecting hom general}
   Let $W  \subset N_\R$ be a tropical fan. 
   Since $\F^{p}_{W}$ is a constructible sheaf, 
    we have 
          $$H^{p, q}_{\Trop} ( W) 
          \cong 
         \begin{cases}
            \F^{p}_{W,  0_N } 
             \cong 
                   F^{p}_W ( \{ 0_N \} )  &  (q=0) \\ 
                      0 & (q \neq 0) 
          \end{cases} $$
      (\cite[Proposition 3.1]{JellShawSmackaSuperformstropicalcohomologyandPoincarduality2019}), 
      where $0_N \in N$ is the zero element.
    Then we have a commutative diagram 
    $$\xymatrix{
      H^{p+1,0}_{\Trop} (W  \times \R) \ar[r]^-{\delta_{W \times \R }} \ar[d]^-{\cong} & 
      H^{p+1,1}_{\Trop, W \times \{ \infty \}} 
      (W  \times (\R \cup \{ \infty \}) ) \ar[d]^-{\cong} &
       \\ 
      H^{p,0}_{\Trop} (W  ) 
      \times  H^{1,0}_{\Trop} ( \R) \ar[r] &
      H^{p,0}_{\Trop} (W ) 
      \times 
      H^{1,1}_{\Trop, \{ \infty \}} ( \R \cup \{ \infty \})
       \ar[r]^-{\cong} &
      H^{p,0}_{\Trop} (W  ), 
    }$$
    where  $\delta_{W \times \R} $ is the connecting homomorphism, 
    and the last horizontal arrow is given by the isomorphism (\ref{eq. signs H^1,1 R }).

    The composition 
    $$ F^{p+1}_{  W \times \R } ( \{ 0_{N \times \R} \} ) \cong 
      H^{p+1,0}_{\Trop} (W  \times \R) 
      \to  H^{p,0}_{\Trop} (W  ) \cong 
     F^{p}_W ( \{ 0_{N} \}  ) $$
     is given by 
  \begin{align*}
      m_1 \wedge \dots \wedge m_{p}
       \wedge x
       \mapsto & 
      m_1 \wedge \dots \wedge m_{p}\\
      m_1 \wedge \dots \wedge m_{p+1} \mapsto & 0
  \end{align*}
   ($m_i \in M $), where  
      $0_{N \times \R} \in N \times \R$ is the zero element, 
      and 
     $ x $ is the coordinate function of $\R \cup \{ \infty \}$, considered as an element of $\F^1 (\R) = F^1 (\R)$.
     This is because the first vertical arrow 
     $$ H^{p+1,0}_{\Trop} (W  \times \R) \to
      H^{p,0}_{\Trop} (W  ) 
      \times  H^{1,0}_{\Trop} ( \R), $$
     is given by multiplying $(-1)^p$ on $F^*$ by (\ref{eq. Kunneth F level}), 
     and a horizontal arrow 
     $$ H^{p,0}_{\Trop} (W  ) 
      \times  H^{1,0}_{\Trop} ( \R) 
      \to 
      H^{p,0}_{\Trop} (W ) 
      \times 
      H^{1,1}_{\Trop, \{ \infty \}} ( \R \cup \{ \infty \})$$
    is given by $(-1)^p \Id \times \delta_{\R}$ 
    by our sign convention (\ref{eq sign of tensor product}). 
    (Note that the connecting homomorphism $\delta_\R$ is given by the differential of a resolution of $\F^1_{\R \cup \{\infty\}}$.)
  \end{exam}

 \subsection{Amini-Piquerez's resolutions and Gysin maps}\label{sub:resolutions_and_gysin_maps}
  In this subsection,  we shall recall explicit complexes of tropical cohomology given by Amini-Piquerez in \cite{AminiPiquerezHomologyoftropicalfans2021}, \cite{AminiPiquerezTropicalFeichtner-Yuzvinskyandpositivitycriterionforfans2024}.
  It is similar to complexes for singular cohomology of linear varieties in \cite{TotaroChowgroupsChowcohomologyandlinearvarieties2014}  and similar to Gersten resolutions (see \cite{MikamiOntropicalcycleclassmaps2020}). 
    We also recall Gysin maps given in \cite[Remark 3.15]{AminiPiquerezHodgetheoryfortropicalvarieties2020}.

  Let $T_{\Sigma}$ be the smooth toric variety of dimension $n$ 
    corresponding to a fan $\Sigma$ in $N_\R$.
  Let $ (X, w )$ be a tropical fan of dimension $d$ in $\Trop (T_{\Sigma})$
    regular at infinity.
 
    We put 
    $$X^p:= \bigcup_{\sigma \in \Sigma_p} X \cap \overline{N_{\sigma,\R}} . $$
    Let $r\geq 0$.
   We have distinguished triangles 
   \begin{align}
     Rj_{l+1*}j_{l+1}^! \F^r_X \to Rj_{l*}j_l^! \F^r_X \to Ri_{l+1,*} i_{l+1}^* Rj_{l*}j_l^! \F^r_X \to^{[1]}   
     \label{eq. distinguished triangles inducing spectral sequence}
    \end{align}
   on $X$ ($l \geq 0$), 
   where $j_l \colon X^l \to X$ and $i_l \colon X \setminus X^l \to X$ 
    are inclusions.
  This  induces exact couples, and hence a spectral sequence 
    $$ E_1^{p,q} =  \bigoplus_{ \sigma \in \Sigma_p } 
         H^{r, p+q-r}_{\Trop, X \cap N_{\sigma,\R} } 
          (X \setminus X^{p+1})   
         \Rightarrow H^{r,p+q-r}_{\Trop} (X) .$$

     By the K\"{u}nneth formula (\ref{eq. Kunneth formula}) 
     and the isomorphism (\ref{eq. signs H^1,1 R }), we have 
    \begin{align*}
         H^{r, p+q-r}_{\Trop, X \cap N_{\sigma,\R}  } 
          (X \setminus X^{p+1})   
          \cong  &
          H^{r-p, q-r}_{\Trop} ( X \cap N_{\sigma,\R}) 
          \otimes 
          H^{p, p}_{\Trop, \{ \infty \}^p } 
            (  ( \R \cup \{ \infty \} )^p ) \\
          \cong &
          H^{r-p, q-r}_{\Trop} ( X \cap N_{\sigma,\R}) 
            .
    \end{align*}
  Since $X \cap N_{\sigma,\R}$ is a fan in $N_{\sigma,\R}$ 
    and 
    $\F^{r-p}_X |_{X \cap N_{\sigma,\R}} 
       = \F^{r-p}_{X \cap N_{\sigma,\R}}$ is a constructible sheaf, 
    we have 
          $$H^{r-p, q-r}_{\Trop} ( X \cap N_{\sigma,\R}) 
          \cong 
         \begin{cases}
            \F^{r-p}_{X \cap N_{\sigma,\R} , 0_{\sigma} } 
             \cong 
                   F^{r-p}_{ X \cap N_{\sigma,\R} } 
                    ( \{ 0_\sigma \})  &  (q=r) \\ 
                      0 & (q \neq r) 
          \end{cases} $$
      (\cite[Proposition 3.1]{JellShawSmackaSuperformstropicalcohomologyandPoincarduality2019}),
     where $0_\sigma \in N_{\sigma,\R}$ is the zero element. 
      In particular, 
    we have $E_1^{p,q} =0 $ for $q \neq r$.

  Thus we get Amini-Piquerez's complex 
   \begin{align} \label{eq AP's complex}
    C^{r,*-r}_{\AP, X } 
     :=  \bigoplus_{ \sigma \in \Sigma_{*-r} }
     F^{r-(*-r)} (0_\sigma, X \cap N_{\sigma,\R}) 
     \end{align}
     whose 
    $(r+p)$-th cohomology is isomorphic to 
         $ H^{r,p}_{\Trop} (X) $, 
         and  
  differentials are 
  given by 
  \begin{align*}
    F^{r-q} (0_\tau, X \cap N_{\tau,\R}) 
      \to & F^{r-q-1} (0_\sigma, X \cap N_{\sigma,\R}) \\
      m_1 \wedge \dots \wedge m_{r-q-1}
       \wedge x
       \mapsto & 
      m_1 \wedge \dots \wedge m_{r-q-1}\\
      m_1 \wedge \dots \wedge m_{r-q} \mapsto & 0
  \end{align*}
   ($\tau \in \Sigma_q$, $\sigma \in \Sigma_{q+1}$ with $\tau \subset \sigma$, 
    $m_i \in M \cap \sigma^{\perp}$)
         by \Cref{exam: connecting hom general} and our sign convention (\ref{eq sign of tensor product}), 
    where
    $x \in M \cap \tau^{\perp}$ is an element such that $v_l(x)=1$ 
    for  the primitive vector $v_l \in N$ of  the $1$-dimensional cone  $l$
      such that $\sigma = l + \tau$. 

  For $\sigma \in \Sigma_s$, 
   we have a natural morphism 
    $$  C^{r,*-r}_{\AP, X \cap \overline{N_{\sigma,\R}}} 
        \to  
        C^{r+s,*-r-s}_{\AP, X } [2s]  
        $$
    of complexes
    given by the identity map of 
    each $F^i (0_\tau, X \cap N_{\tau,\R})$ 
    ($\tau \in \Sigma$ containing $\sigma$).
  This morphism induces a morphism 
    $$  H^{r, p }_{\Trop} ( X \cap \overline{N_{\sigma,\R}} ) 
    \to H^{r+s, p+s}_{\Trop} 
    ( X ), $$
    the \emph{Gysin map} 
    (\cite[Remark 3.15]{AminiPiquerezHodgetheoryfortropicalvarieties2020}).
  This is compatible with Gysin maps of Chow groups of toric varieties, see [loc.cit.].

\section{Log tropical cohomology and spectral sequences}
  \label{sec:log_tropical_cohomology_and_spectral_sequences}
  In this section, we shall prove the main theorem 
  (Corollary \ref{spectral sequence for ss reduction}), 
  a new construction of 
  the spectral sequences given by Itenberg-Katzarkov-Mikhalkin-Zharkov \cite{ItenbergKatzarkovMikhalkinZharkovTropicalhomology2019} 
  (in the realizable and smooth case)
  and Amini-Piquerez \cite{AminiPiquerezHodgetheoryfortropicalvarieties2020}  
  (in the smooth case)
  without the smoothness assumption.
  Similar spectral sequences are given in  \cite{MikamiOntropicalcycleclassmaps2020} 
   for tropical cohomology of geometric semi-stable reductions of algebraic varieties. 

   Our construction is 
     based on relative log tropical homolomorphic forms,
     and is 
   the same as that of Steenbrink's geometric monodromy-weight spectral sequences 
   for 
  limit mixed Hodge structures \cite[Corollary 4.20]{SteenbrinkLimitsofHodgestructures1976}, \cite[Chapter 11]{PetersSteenbrinkMixedHodgestructures2008}.

  In Subsection \ref{subsec log tropical cohomology}, we shall introduce log tropical holomorphic forms for tropical fans, 
  and use them to give analogs of weight spectral sequences given by Deligne 
  (\cite{DeligneThoriedeHodgeII1972}, 
  \cite[Chapter 4]{PetersSteenbrinkMixedHodgestructures2008}). 
  This is a generalization of Amini-Piquerez's resolution 
  (\cite{AminiPiquerezHomologyoftropicalfans2021},  \cite{AminiPiquerezHomologicalsmoothnessandDeligneresolutionfortropicalfans2024})
  for smooth tropical fans. 
  In Subsection \ref{sub:semi_stable_reductions_of_tropical_varieties}, 
   we introduce semi-stable reductions of tropical varieties. 
   In Subsection \ref{sub:monodromy_weight_spectral_sequences}, we shall prove the main theorem.
  In Subsection \ref{sub:eigenwaves}, 
   we also show that the eigenwave action is a tropical analog of the residue of Gauss-Manin connection (\Cref{eigenwave is residue of GM connections:label}) for smooth projective tropical varieties.

 \subsection{Log tropical cohomology}\label{subsec log tropical cohomology}
  In this subsection, we study log tropical holomorphic forms and analogs of Deligne's weight spectral sequences
  (\cite{DeligneThoriedeHodgeII1972}, 
  \cite[Chapter 4]{PetersSteenbrinkMixedHodgestructures2008})
   for tropical fans, 
  generalizing Amini-Piquerez's one 
  (\cite{AminiPiquerezHomologyoftropicalfans2021},  \cite{AminiPiquerezHomologicalsmoothnessandDeligneresolutionfortropicalfans2024})
  for smooth tropical fans.

  Let $\Sigma$ be a unimodular fan in $N_\R$.
  Let $X \subset \Trop (T_\Sigma)$ be a tropical fan regular at infinity.
  Let $I \subset \Sigma$ be a subset of $1$-dimensional cones.
  We put $D:= X\cap \bigcup_{l\in I} \overline{N_{l,\R}}$, a closed subset of $X$,
  and $U:= X \setminus D$ the complement.
  We put $D_l := X \cap \overline{N_{l,\R}}$ ($l \in I$).

   We put $\F^r_X (\log D) :=i_* \F^r_U$, where $i \colon U \hookrightarrow X$ is the inclusion.
   The cohomology $ H^{r+q} (X, \F^r_X (\log D) [-r])   $ 
   is called 
   \emph{log tropical cohomology} of $(X,D)$.
  Since $X$ is regular at infinity, 
  by the K\"{u}nneth formula (\ref{eq. Kunneth formula}), 
  we have  
  $R^i i_* \F^r_U  =0$
   ($i \geq 1$).
  Hence  we have 
  $$ H^{r+q} (X, \F^r_X (\log D)[-r]) \cong H^{r,q}_{\Trop} (U).$$

  Let $x \in X$ be a point. We put
   $I_x := \{ l \in I  \mid x \in D_l \}$.
  We put $\sigma_x \in \Sigma$ the cone such that $x \in N_{\sigma_x,\R}$.
  Then we can identify 
  $$\Trop (T_{\sigma_x}) 
     \cong (\R \cup \{ \infty \})^{\dim \sigma_x  }
      \times \R^{n -\dim \sigma_x }$$
  with 
  $I_x = \{  \R_{\geq 0} \cdot e_i \}_{1\leq i \leq \# I_x }$, 
  where 
   $T_{\sigma_x} \subset T_{\Sigma}$ is the affine open toric subvariety corresponding to the cone $\sigma_x \in \Sigma$, 
   and 
  $e_i $ ($1\leq i \leq n$) is the $i$-th coordinate 
   of $\R^n \cong N_\R$.
  We put $f_i \in M$ ($1\leq i \leq n$) the dual basis of $e_i$.
  For $l =\R_{\geq 0} \cdot e_i\in I_x$,
  we put $f_l:= f_i$.
  Then we have an isomorphism 
  \begin{equation}
  \begin{aligned}
  \bigoplus_{ i=0}^r 
  \F^{r-i}_{X \cap N_{\sigma_x,\R},x}
   \otimes  
  \bigwedge^i \Q \langle f_l \rangle_{l \in I_x}  
  & 
   \cong (\F^r_X (\log D))_x  
   \\
  f  \otimes  f_{l_1}\wedge \dots \wedge f_{l_{i}} 
  & \mapsto
    f \wedge f_{l_1}\wedge \dots \wedge f_{l_{i}} 
   \label{log Fp local description} .
  \end{aligned}
  \end{equation}
  ($l_j \in I_x $ , $f \in \F^{r-i}_{X \cap N_{\sigma_x,\R},x}$).

  Let $J:= \{j_1,\dots,j_{\# J}\} \subset I$ be a subset such that $D_J := \bigcap_{j\in J} D_j \neq \emptyset$.
  We put $D_J \cap D := D_J \cap \big(\bigcup_{l \in I \setminus J} D_l \big)$ a closed subset of $D_J$.
  We put $j_J \colon D_J \hookrightarrow X$ the inclusion.
  For $ j \in I$, we have a natural morphism
  $$\F^r_X (\log D) \to j_{\{j\}*} \F^{r-1}_{D_{j}} (\log (D_{j} \cap  D ) )   $$
  which is given, at $x \in D_{j}$, by 
  \begin{align*}
  \F^r_X (\log D)_x 
  \cong \bigoplus_{ i=0}^r 
  \F^{r-i}_{X \cap N_{\sigma_x,\R},x}  
   \otimes  
  \bigwedge^i \Q \langle f_l \rangle_{l \in I_x}  
  & 
  \to 
  \bigoplus_{ i=0}^{r-1} 
  \F^{r-i-1}_{X \cap N_{\sigma_x,\R},x}  
    \otimes  
  \bigwedge^i \Q \langle f_l \rangle_{l \in I_x \setminus \{ l_0 \}}  
  \\
  & \cong \F^{r-1}_{D_{l_0}} (\log (D_{l_0} \cap  D ) )_x  
  \\
  f \wedge f_{l_1}\wedge \dots \wedge f_{l_{i-1}} \wedge f_{j}  
  & \mapsto
   f \wedge f_{l_1}\wedge \dots \wedge f_{l_{i-1}} 
    \\
   f \wedge f_{l_1}\wedge \dots \wedge f_{l_i}   & \mapsto 0
  \end{align*}
  ($l_j \in I_x \setminus \{j \}$ , $f \in \F^{r-i}_{X \cap N_{\sigma_x,\R},x}$).
  We fix a total order $o(J)= (j_1 < \dots < j_{\# J})$ 
   for each $J \subset I$.
  We get a composition of morphisms 
  \begin{align*}
  \Res_{o(J)} \colon 
    \F^r_X (\log D)&  \to j_{\{j_1\}*} \F^{r-1}_{D_{j_1}} (\log (D_{j_1} \cap D ) )  
    \\
  & \to 
  j_{\{j_1,j_2\}*} \F^{r-1}_{D_{\{j_1,j_2\}}} (\log (D_{\{j_1,j_2\}} \cap D ) )
  \\ 
  & \to \dots \to 
  j_{J*} \F^{r-\# J}_{D_J} (\log (D_J \cap D ) ).
  \end{align*}
  We put 
  \begin{align*}
   W_s \F^r_X (\log D) 
  :=  \Ker \bigg( 
  (\Res_{o(J)})_J
   \colon 
  \F^r_X (\log D)  
  \to 
  \bigoplus_{\substack{J \subset I \\ \# J= s+1}}
  j_{J,*} \F^{r-s-1}_{D_J} (\log (D_J \cap D ) )
  \bigg),
  \end{align*}
  which does not depend on orders.
  This is an increasing filtration, called the \emph{weight filtration}.
  For example, 
  $W_0 \F^r_X (\log D) = 
   \F^r_X  $.
  We put 
  $$ 
   \Gr^W_s \F^r_X (\log D) 
  := W_s \F^r_X (\log D)  / W_{s-1} \F^r_X (\log D) .$$
    For $s \geq 0$, 
    we have an isomorphism
   $$
  (\Res_{o(J)})_J [-r]
   \colon 
   \Gr^W_s \F^r_X (\log D) [-r]
  \cong \bigoplus_{\substack{J \subset I \\ \# J= s}}
  j_{J ,* } \F_{D_J  }^{r-s} [-r].$$
   Thus we get the following.

  \begin{prp}\label{weight spectral sequence for smooth open varieties}
    For $r \geq 0$, 
    the $E_1$-terms of the spectral sequence 
    $$E_1^{p,q}= H^{p+q} (X, \Gr^W_{-p} \F_{X}^r (\log D)[-r ])  \Rightarrow 
   H^{p+q} (X, \F^r_X (\log D)[-r])$$
      induced 
      by the weight filtration are isomorphic to
    $$ \bigoplus_{\substack{J \subset I \\ \# J= -p}}
        H^{p+q} (D_J, \F_{D_J  }^{r+p} [-r])
        \cong 
        \bigoplus_{\substack{J \subset I \\ \# J= -p}}
          H^{r+p, p+q-r }_{\Trop} ( D_J ) .$$
  \end{prp}
  
  In the rest of this subsection,
   we shall show that 
   the differential 
  $d_1 \colon E_1^{p,q} \to E_1^{p+1,q}$
   in Proposition \ref{weight spectral sequence for smooth open varieties} 
   is given by Gysin maps (Subsection \ref{sub:resolutions_and_gysin_maps}) up to signs (\Cref{sign of log spectral sequence:label}) 
   (cf. \cite[Proposition 1.8]{GuillenNavarroSurlethormelocaldescyclesinvariantsAznar1990},   \cite[Proposition 4.7 and 4.10]{PetersSteenbrinkMixedHodgestructures2008}).
   The differential 
    $d_1 \colon E_1^{p,q} \to E_1^{p+1,q}$
   is  the connecting homomorphism of 
    a distinguished triangle
   \begin{align*}
           \Gr^W_{-p-1} \F_{X}^r (\log D) [-r]
            \to  &
          (W_{-p} \F_{X}^r (\log D)/ W_{-p-2} \F_{X}^r (\log D))
          [-r]  \\
         \to  &
          \Gr^W_{-p} \F_{X}^r (\log D) [-r] \to^{[1]}.
   \end{align*}

   To simplify notation, we put 
     $$ W_{-p} \F_{X}^r (\log D)/ W_{-p-2}
      :=
      W_{-p} \F_{X}^r (\log D)/ W_{-p-2} \F_{X}^r (\log D). 
      $$
  For $J' \subset J  \subset I$ 
     with $\# J'  =-p-1 $ and $\# J=-p$
     and orders $o(J') = ( j_1'< \dots < j_{-p-1  }' )$ 
         and $o(J)= (j_1 < \dots < j_{-p})$, 
  we define
    $\sgn (o(J'),o(J)) \in \{ \pm 1 \}$
    as 
   the sign of the permutation 
   that sends
    $ ( j_1', \dots , j_{-p-1}' , j) $ 
    to $ (j_1 , \dots , j_{-p})$,
    where $j \in J$ is such that $J= \{ j \} \cup J'$.

  We shall consider another topology of $X$
   whose 
   open subsets are unions of $X \cap \Trop (T_\sigma)$ ($ \sigma \in \Sigma$), 
   and coverings are set-theoretical coverings.
   Let $X_{\Zar}$ denote $X$ equipped with this topology.
  The identity map of $X$ induces a continuous map
   $\pi \colon X \to X_{\Zar}$, 
   where let $X$ denote the set $X$ equipped with the usual topology. 
   By using 
   Amini-Piquerez's complex
    $C^{r,*-r}_{\AP, X \cap T_{\sigma} } $ ($ \sigma \in \Sigma$) 
    (Subsection \ref{sub:resolutions_and_gysin_maps}), 
    we can easily see that 
   $R\pi_*\F^r \cong \pi_* \F^r$.
   Amini-Piquerez's complex
    $C^{r,*-r}_{\AP, X } $ gives  
   a flabby resolution 
     $\C^{r,*-r}_{\AP, X } $ of $\pi_*\F^r[-r]$ 
     by considering 
     $  F^{r-(*-r)}_{X \cap N_{\sigma, \R}} ( \{ 0_\sigma \} ) $
     as 
     a skyscraper sheaf at $0_{\sigma}$.

  We shall use Amini-Piquerez's complexes to compute the connecting homomorphism.
  We put 
   $$\C_{\AP, \Gr_j^W}^{r,*-r}
   := 
   \bigoplus_{\substack{J \subset I \\ \# J= j}}
   \C_{\AP, D_J }^{r-j, *-r+j } [-j] , 
   $$
   a flabby resolution of 
   $ \pi_* \Gr^W_j \F^r_X (\log D)[-r] $. 
  We also put 
   $\C_{\AP, W_{-p} / W_{-p-2}}^{r,*-r} $
   the flabby complex whose $q$-th term ($q \geq r$)
   is 
   $$\C_{\AP, \Gr_{-p}^W}^{r,q-r} \oplus 
    \C_{\AP, \Gr_{-p-1}^W}^{r,q-r}$$
   and, 
   differentials 
   are given by the sum of the differentials of 
    $\C_{\AP, \Gr_{-p}^W}^{r,*-r}  $
    and $  \C_{\AP, \Gr_{-p-1}^W}^{r,*-r} $ 
   and the identity map 
    $ \times  (-1)^{p} \cdot \sgn (o(J'),o(J))$
   of $F^{2r+p-q}_{ X \cap N_{\sigma_K, \R} } ( \{ 0_{\sigma_K} \} )$ 
    for $J' \subset J \subset K \subset I$ 
     with $\# J'  =-p-1 $, $\# J=-p$, and $\# K = -p+q-r $,
     where $\sigma_K := \sum_{l \in K} l   \in \Sigma$ is the cone corresponding to $K \subset I$.

     Since 
   $ R^i \pi_* \Gr^W_j \F^r_X (\log D) = 0 $ ($i \geq 1$), 
   we have 
    $ R^i \pi_* W_{-p} \F_{X}^r (\log D)/ W_{-p-2} =0 $ ($i \geq 1$).
   We have a quasi-isomorphism 
    $$
    \pi_* W_{-p} \F_{X}^r (\log D)/ W_{-p-2}  [-r]
    \cong 
      \C_{\AP, W_{ -p } / W_{-p-2}}^{r,*-r} $$
   given as the sum of a natural morphism 
    $$ 
    \pi_* W_{-p} \F_{X}^r (\log D)/ W_{-p-2} [-r] 
    \to 
     \pi_* \Gr^W_{-p} \F_{X}^r (\log D) [-r] 
    \cong 
      \C_{\AP, \Gr_{-p}^W}^{r,*-r} $$
      of complexes
   and a morphism
   \begin{align*}
      (\Res_{o(J)})_{\substack{J \subset I \\ \# J= -p-1}}
      \colon 
      \pi_* W_{-p} \F_{X}^r (\log D)/ W_{-p-2}    
    \to  \C_{\AP, \Gr_{-p-1}^W}^{r,r-r} 
   \end{align*}
   of sheaves.
    (Note that $ 
      (\Res_{o(J)})_{\substack{J \subset I \\ \# J= -p-1}}
    $ does not extend to a morphism of complexes.)

    Hence we have a commutative diagram
     $$\xymatrix{
        R \pi_*  \Gr^W_{-p-1} \F_{X}^r (\log D) 
          \ar[r] \ar[d]^-{\cong}  &
        R \pi_* W_{-p} \F_{X}^r (\log D)/ W_{-p-2}  
         \ar[r] \ar[d]^-{\cong} &
        R \pi_* \Gr^W_{-p} \F_{X}^r (\log D)   
          \ar[d]^-{\cong} 
          \\
        \C_{\AP, \Gr_{-p-1}^W}^{r,*-r} [r]
        \ar[r] &
       \C_{\AP, W_{-p} / W_{-p-2}}^{r,*-r} [r]
       \ar[r] &
      \C_{\AP, \Gr_{-p}^W}^{r,*-r} [r]
      .
     }$$
     The second row extends to a short exact sequence, 
     and hence we can compute connecting homomorphisms of the shift of the first row using differentials of 
     $ \C_{\AP, W_{-p} / W_{-p-2}}^{r,*-r} $.

     Consequently, we have the following (cf. 
      \cite[Proposition 1.8]{GuillenNavarroSurlethormelocaldescyclesinvariantsAznar1990},  
   \cite[Proposition 4.7 and 4.10]{PetersSteenbrinkMixedHodgestructures2008}).
   Recall that 
   Gysin maps 
   (Subsection \ref{sub:resolutions_and_gysin_maps})
   are given by identity maps of direct summands of Amini-Piquerez's complexes.

  \begin{cor}\label{sign of log spectral sequence:label}
   The differential 
   $d_1 \colon E_1^{p,q} \to E_1^{p+1,q}$
   in Proposition \ref{weight spectral sequence for smooth open varieties}
   is given by   $  (-1)^{p} \cdot  \sgn (o(J'), o(J)) \times $ Gysin maps
    $$  H^{r+p, p+q -r}_{\Trop} ( D_J ) 
          \to 
          H^{r+p+1, p+1+q -r }_{\Trop} ( D_{J'} ) $$
    ($J' \subset J \subset I $ 
       with $ \# J'= -p-1$ and $ \# J= -p$). 
  \end{cor}

 \subsection{Semi-stable reductions of tropical varieties}
  \label{sub:semi_stable_reductions_of_tropical_varieties}
  Our construction of the spectral sequences in the main theorem (\Cref{introduction spectral sequence for ss reduction})
  is based on \emph{strictly semi-stable reductions} of tropical varieties, introduced in this subsection. 
   They are defined by using cones over tropical varieties, 
   which are used 
   in the original work of  Itenberg-Katzarkov-Mikhalkin-Zharkov \cite{ItenbergKatzarkovMikhalkinZharkovTropicalhomology2019} 
   and also by Gubler to study
     tropical compactifications over non-trivially valued fields.
  Strictly semi-stable reductions of tropical varieties might be known for some specialists.
    However, to the author's knowledge, there seems to be no explicit reference in the literature. 

  Let $\Sigma$ be a complete unimodular fan in $N_\R$,
  and $(X,w) $
    a tropical variety in $ \Trop (T_{\Sigma})$ regular at infinity. 
  Since we discuss tropical cohomology 
    with rational coefficients, 
   by  \Cref{the existence of unimodular triangulation}, 
   without loss of generality, 
   we may assume that there is 
   $N$-unimodular
    $\Q$-rational 
    polyhedral complex strucuture $\Lambda$ 
    of the tropical variety $(X,w)$.
    Then by unimodularity (\Cref{def unimodularity:label}), $\rec ( \Lambda \cap N_\R ):= \{ \rec (P \cap N_\R  ) \}_{P \in \Lambda }$ forms a subfan of $\Sigma$,
   where $\rec(P \cap N_\R )$ is the unique cone such that 
       $P  \cap N_\R  $ is the Minkowski sum of 
        $\rec(P \cap N_\R )$ and a compact polyhedron.

   We consider a new $\Z$-module $N \times \Z$.
  Positive real numbers $\R_{>0}$ acts on $\Trop (T_{\Sigma}) \times \R$ by the component-wise multiplication. 

  \begin{dfn}\label{fans over polyhedral complex and quasi-projectivity:label}
   Let $\Xi $ be a $\Q$-rational polyhedral complex in $\Trop ( T_{\Sigma} )$.
    When  the closure of 
    $\R_{\geq 0} \cdot  
    ( \lvert \Xi \rvert \times \{1 \} )$
   in $\Trop ( T_{\Sigma} ) \times \R$
   has a (unique) fan structure $\tilde{\Xi}$
   such that 
   $$ \Xi  \times \{1 \} 
      = \{ \tilde{P} \cap (\Trop ( T_{\Sigma} ) \times \{ 1 \})
       \}_{\tilde{P} \in \tilde{\Xi}} ,
        $$
    we call $\tilde{\Xi}$ the fan over $\Xi$.
  \end{dfn}

  \begin{rem}
  When $\lvert \Xi \rvert \cap N_\R \subset \lvert \Xi \rvert $ is dense, 
   there exists such a fan $\tilde{\Xi}$
   if and only if 
   $\rec (\Xi \cap N_\R) $ forms a fan.
   In particular, there is a fan $\tilde{\Lambda}$ 
    over $\Lambda$.
  \end{rem}

  We put 
  $$\tilde{\Lambda} |_{ N_\R \times \R } 
     := \{ \tilde{P} \cap (N_\R \times \R) \}_{\tilde{P} \in \tilde{\Lambda}}, $$
     a fan in $N_\R \times \R$,
   and put 
   $\Psi \colon \Trop (T_{\tilde{\Lambda} |_{ N_\R \times \R }})\to \R \cup \{ \infty \} $
   the natural extension of the projection 
   $ N_\R \times \R \to \R$.
   We also put 
   $$ \X := \overline{ \R_{>0 } \cdot ( X\times \{ 1  \} )}
   \subset \Trop (T_{\tilde{\Lambda}|_{  N_\R \times \R } })   $$
   the closure.  Obviously, $\X $ is regular at infinity in 
   $\Trop (T_{\tilde{\Lambda}|_{  N_\R \times \R } })   $.
   We put $I 
     \subset \tilde{\Lambda} |_{ N_\R \times \R } 
     $ the subset of $1$-dimensional cones $l$ with $\Psi (l) = \R_{\geq 0}$.
     For each $J \subset I$, 
     we put 
     $$ X_J: = 
         \X \cap \bigcap_{ l \in J  }
           \overline{ ( N \times \Z )_{ l , \R } }
          =  
          \bigcup_{ \bigcup_{l \in J} l \subset P \in \tilde{ \Lambda }|_{ N_\R \times \R , \dim X +1}
           }
          \overline{P} \cap 
         \bigcap_{ l \in J  }
           \overline{ ( N \times \Z )_{ l , \R } }.
           $$
    We put
      $$w_J  |_{ \relint (
        \overline{P} \cap 
         \bigcap_{ l \in J  }
           \overline{ ( N \times \Z )_{ l , \R } }
                )  }
            := w_P. 
      $$
    We endow $X_J$ with a sturucture of tropical variety $(X_J, w_J)$ of dimension ($\dim X - \# J +1$).
    To simplify notation, for $l \in I$, 
      we put $X_l:= X_{ \{ l \} }$.
   The set 
   $$ X_{\infty} := \X \cap \Psi^{-1} (\infty) 
         = \bigcup_{ l \in I } X_l
         $$
    is also a tropical variety (by $w_{ \{ l \} }$ ($l \in I$)), 
    called a \emph{strictly semi-stable reduction} of $X$.

  \begin{rem}
   In the realizable case, reductions of tropical varieties 
    are compatible with those of algebraic varieties 
    under tropicalizations.
   This is essentially done in \cite{GublerAguidetotropicalizations2013}. 
  \end{rem}

  \begin{rem}\label{quasi-projectivity of polyhedral complex structures}
  We assume that $\Sigma$ is projective.
  Then
   $\tilde{\Lambda} |_{ N_\R  \times \R}$ is 
   finer than a quasi-projective fan
   $\{ \sigma \times \{ 0 \}, \sigma \times \R_{\geq 0} \}_{
     \sigma \in \Sigma }$,
     and hence is 
   unimodular quasi-projective fan. 

   In particular, 
     an ample line bundle $l \in \CH^1 (T_\Xi) \cong H^{1,1}_{\Trop} (\Trop (T_\Xi))$
     of a smooth projective toric compactification $T_\Xi$
      of $T_{\tilde{\Lambda} |_{   N_\R \times \R } } $ 
      gives a K\"{a}hler class of 
    $(X,w, \Lambda)$  
    in the sense of \cite[Definition 6.5]{AminiPiquerezHodgetheoryfortropicalvarieties2020}. 
    A K\"{a}hler class is needed to give 
    a bigraded (Hodge)-Lefschetz structure of $H^{p,q}_{\Trop}(X)$ when $X$ is smooth (see [loc.cit.] and also \cite{GuillenNavarroSurlethormelocaldescyclesinvariantsAznar1990}).
  \end{rem}

 \subsection{Monodromy-weight spectral sequences}\label{sub:monodromy_weight_spectral_sequences}
  In this subsection, 
  we shall prove the main theorem 
  (\Cref{spectral sequence for ss reduction}). 
  Namely, we shall give  
  a new construction of 
  the spectral sequences given by Itenberg-Katzarkov-Mikhalkin-Zharkov \cite{ItenbergKatzarkovMikhalkinZharkovTropicalhomology2019} 
  (in the realizable and smooth case)
  and Amini-Piquerez \cite{AminiPiquerezHodgetheoryfortropicalvarieties2020}  
  (in the smooth case)
  in the same way as Steenbrink \cite{SteenbrinkLimitsofHodgestructures1976} (see also \cite[Chapter 11]{PetersSteenbrinkMixedHodgestructures2008})
  using relative log tropical homolomorphic forms.

  We shall use notation in the previous subsection.

   First, we shall show that $H_{\Trop}^{p,q} (X)$ is isomorphic to cohomology of a relative log sheaf $\F^p_{\X / \R \cup \{ \infty  \} } (\log X_{\infty} )|_{ X_{\infty}  }$ on $X_{\infty}$ (Proposition \ref{ss reduction trop coh = log coh}).
  We put $\X^{\circ} :=\X \setminus \Psi^{-1}( 0  )$ 
  and $i \colon \X \setminus X_{\infty} 
  \hookrightarrow \X$.
   Since $\X $ is regular at infinity, 
  we can easily see that 
  $R^i i_* \F^p_{\X  \setminus X_{\infty}} = 0$ ($i \geq 1$). 
  Hence we have 
  \begin{align}
   H^{q} ( \X^{\circ} \setminus X_{\infty} ,
    \F^p_\X |_{\X^{\circ} \setminus X_{\infty}}) 
   \cong H^{q} ( \X^{\circ}, 
   \F^p_{\X} (\log X_{\infty} ) |_{ \X^{\circ}}  ),
    \label{eq tro coh relative log1}
  \end{align}
  where we put $
   \F^p_{\X} (\log X_{\infty} ) := i_* \F^p_{\X \setminus X_{\infty}}  $.
  We put 
  $$ \F^1_{\R \cup \{ \infty  \}} (\log \{ \infty  \}) 
    \wedge \F^{p-1}_\X (\log X_{\infty})$$
  the subsheaf of $\F^p_\X (\log X_{\infty})$
  generated by the image of 
  $ \Psi^{*} \F^1_{\R \cup \{ \infty  \}} (\log \{ \infty  \}) 
    \otimes \F^{p-1}_\X (\log X_{\infty})$
    under the wedge product.
  We put $t := (0,1)\in M\times \Z =\Hom (N \times \Z,\Z)$.
  
  \begin{lem}\label{Fp log computation}
    For $p \geq 0$, there is an exact sequenece 
    \begin{align*}
    0 \to \F^1_{\R \cup \{ \infty  \}} (\log \{ \infty  \}) 
     \wedge \F^{p-1}_\X (\log X_{\infty}) |_{ \X^{\circ} }
       \to &  \F^p_\X (\log X_{\infty}) |_{ \X^{\circ} } \\
       \to & \F^1_{\R \cup \{ \infty  \}} (\log \{ \infty  \}) 
     \wedge \F^{p}_\X (\log X_{\infty}) |_{ \X^{\circ} }
       \to 0 , 
    \end{align*}
    where the third morphism is given by $ t \wedge -$,
    and 
    we put 
    $$\F^1_{\R \cup \{ \infty  \}} (\log \{ \infty  \}) 
     \wedge \F^{-1}_\X (\log X_{\infty}) :=0 .$$
  \end{lem}
  \begin{proof}
  This follows from (\ref{log Fp local description}).
   (Note also that we have a homeomorphism
  $$\X^{\circ} \setminus X_{\infty}  \cong X \times (0,\infty).)$$
  \end{proof}
  
  By (\ref{eq tro coh relative log1}), 
  Lemma \ref{Fp log computation},
  long exact sequences of cohomology, 
  and induction on $p \geq 0$, 
  we have 
  \begin{align}\label{eq cohomology eq. wedge FR1}
    & H^{q} ( \X^{\circ} \setminus X_{\infty} ,
      \F^1_{\R \cup \{ \infty  \}} 
    \wedge \F^{p}_\X  |_{
      \X^{\circ} \setminus X_{\infty}
    } )  
   \cong  & H^{q} ( \X^{\circ}, 
     \F^1_{\R \cup \{ \infty  \}} (\log \{ \infty  \}) 
    \wedge \F^{p}_\X (\log X_{\infty})  |_{ \X^{\circ} }
     ).
  \end{align}
  We put 
   $$\F^p_{\X / \R \cup \{ \infty  \} } (\log X_{\infty} ) :=
    \F^p_\X (\log X_{\infty}) / 
    \F^1_{\R \cup \{ \infty  \}} (\log \{ \infty  \}) 
    \wedge \F^{p-1}_\X (\log X_{\infty}).$$
  Obviously, we have
  $$ \F^p_{\X / \R \cup \{ \infty  \} } (\log X_{\infty} ) |_{X} \cong \F^p_X.$$
  Hence  by (\ref{eq tro coh relative log1}),
   (\ref{eq cohomology eq. wedge FR1}), 
  and 
  long exact sequences of cohomology, 
  we have 
  $$ 
  H^{p,q}_{\Trop} (X) =
  H^{p+q} ( X, \F^p_X [ -p ]) 
   \cong H^{p+q} ( \X^{\circ} , 
    \F^p_{\X / \R \cup \{ \infty  \} } (\log X_{\infty} ) [ -p ] 
     |_{ \X^{\circ} }). $$

  By using a retraction
     $ \X^{\circ} \to X_{\infty}$ preserving polyhedral complex structures (cf. \cite[Section 3]{JellShawSmackaSuperformstropicalcohomologyandPoincarduality2019}), 
  we also have 
  $$
    H^{p+q} ( \X^{\circ} ,  \F^p_{\X / \R \cup \{ \infty  \} } (\log X_{\infty} )  [ -p ] 
     |_{ \X^{\circ} })
   \cong H^{p+q} ( X_{\infty} ,  
   \F^p_{\X / \R \cup \{ \infty  \} } (\log X_{\infty} ) [-p] |_{ X_{\infty}  })
   .$$

  Consequently, we have the following (cf. \cite[Corollary 2.16]{SteenbrinkLimitsofHodgestructures1976}, \cite[Theorem 11.16]{PetersSteenbrinkMixedHodgestructures2008}).
  \begin{prp}\label{ss reduction trop coh = log coh}
  $$  H^{p,q}_{\Trop} ( X) 
  \cong  H^{p+q} ( X_{\infty} ,  
   \F^p_{\X / \R \cup \{ \infty  \} } (\log X_{\infty} )[-p] |_{ X_{\infty}  }).$$
  \end{prp}

  Next, 
  we shall construct 
   the spectral sequence 
  (\Cref{spectral sequence for ss reduction}) 
   converging to the cohomology 
  $H^{r+q} ( X_{\infty} ,  
   \F^r_{\X / \R \cup \{ \infty  \} } (\log X_{\infty} )[-r] |_{ X_{\infty}  })$.
  We put 
  $$A^{r,*-r} := \F^{*+1}_{\X} (\log X_{\infty} )/ W_{*-r} \F^{*+1}_{\X} (\log X_{\infty} ) |_{X_{\infty}}$$
  ($* \geq r$),
  a complex of sheaves with 
  $$d \colon A^{r,q-r } \ni \alpha \mapsto  \alpha  \wedge t \in A^{r,q-r +1}.$$
  By (\ref{log Fp local description}), 
  we can easily see that 
  a morphism 
  $$\F^r_{\X / \R \cup \{ \infty  \} } (\log X_{\infty} )  |_{ X_{\infty}  }
  \ni \alpha 
  \mapsto   \alpha \wedge t
   \in A^{r,r-r}$$
  induces a quasi-isomorphism 
  $$\F^r_{\X / \R \cup \{ \infty  \} } (\log X_{\infty} ) [-r] |_{ X_{\infty}  } \cong A^{r,*-r}.$$
  We put 
  $$ W(M)_{-p} A^{r,q-r} 
  := \Im (W_{-p+2(q-r) +1}\F^{q+1}_{\X } (\log X_{\infty} ) |_{ X_{\infty}  }  \to A^{r,q-r}). $$
  Since we have $d W(M)_{-p} \subset W(M)_{-p-1}$, 
  we have a subcomplex $W(M)_{-p} A^{r,*-r} $
  of $ A^{r,*-r} $, 
   and the quotient $\Gr^{W(M)}_{-p} A^{r,*-r} $ has zero differential (i.e., $d=0$).
   More presicely, 
   for fixed orders 
   $ o(J)= ( j_1 < \dots < j_{ \# J } )$
   of $J = \{ j_1, \dots, j_{ \# J } \} \subset I$,
  we have 
  \begin{align*}
   (\Res_{ o(J) })_J 
   \colon 
   \Gr^{W(M)}_{-p} A^{r,*-r} 
   \cong &  
    \bigoplus_{\substack{ r \leq u 
               \\   (u-r) +1 \leq -p + 2(u-r) +1 } }
    \Gr^W_{-p +2(u -r) +1} \F_{\X}^{u+1} (\log X_{\infty}) |_{X_{\infty}} 
    [-u] \\
   \cong &  
    \bigoplus_{ \max \{0, p \} \leq u-r  }
    \bigoplus_{\substack{ J \subset I \\ \# J = -p + 2(u-r) +1 } }
     j_{J,*} \F_{X_{J}}^{2r-u+p}  [-u], 
  \end{align*}
  where we put $j_J \colon X_J \hookrightarrow X_{\infty}$ the inclusion (see Subsection \ref{subsec log tropical cohomology}).
  
  Consequently, we get the following.
  \begin{thm}\label{spectral sequence for ss reduction}
    For $r \geq 0$, 
    the $E_1$-terms of the spectral sequence 
    $$E_1^{p,q}= H^{p+q} (X_{\infty}, \Gr^{W(M)}_{-p} A^{r,*-r}  )  \Rightarrow H^{r,p+q-r}_{\Trop } (X)$$
      induced by the monodromy-weight filtration $(W(M)_s A^{r,*-r})_{s } $ are isomorphic to
    $$\bigoplus_{ \max \{0,p\} \leq u -r }
    \bigoplus_{\substack{ J \subset I \\ \# J = -p + 2 (u-r) +1 } }
  H^{2r-u+p, p+q -u}_{\Trop} (X_{J}). $$
  \end{thm}

  Finally, 
   we shall show that 
   the differential $d_1$
   of $E_1$-terms of the spectral sequence 
   in \Cref{spectral sequence for ss reduction}
    is
    given by pull-back maps 
    and 
    Gysin maps up to signs (\Cref{differential of monodromy-weight spectral sequence:label})
    (cf. \cite[Lemme 2.7]{GuillenNavarroSurlethormelocaldescyclesinvariantsAznar1990}, 
     \cite[Subsection 11.3.2]{PetersSteenbrinkMixedHodgestructures2008}).
   The differential $d_1 \colon E_1^{p,q} \to E_1^{p+1,q}$  
   is the  connecting homomorphism 
   \begin{align}\label{eq d1 of weight-monodromy spectral sequences}
     H^{p+q} (X_{\infty}, \Gr^{W(M)}_{-p} A^{r,*-r}  ) 
     \to 
     H^{p+q+1} (X_{\infty}, \Gr^{W(M)}_{-p-1} A^{r,*-r}  ) 
   \end{align}
   of 
   $$ 0 \to 
     \Gr^{W(M)}_{-p-1} A^{r,*-r} 
     \to 
      W(M)_{-p} A^{r,*-r} / W(M)_{-p-2} 
     \to 
      \Gr^{W(M)}_{-p} A^{r,*-r} 
      \to 0, $$
    where to simplify notation, we put 
    $$ 
      W(M)_{-p} A^{r,*-r} / W(M)_{-p-2} 
      :=
      W(M)_{-p} A^{r,*-r} / W(M)_{-p-2} A^{r,*-r} .$$
    For $0 \leq u-r$,
      we have 
    \begin{align*}
      & W (M)_{-p} A^{r,u-r} / W(M)_{-p-2}   \\
      = &
          \begin{cases}
       ( W_{-p+2(u-r)+1}  \F_{\X}^{u+1} (\log X_{\infty}) 
       / W_{-p+2(u-r)-1}  )
        |_{X_{\infty}} 
        &
        (  p \leq u-r -1)
        \\ 
       \Gr^W_{-p+2(u-r)+1}  \F_{\X}^{u+1} (\log X_{\infty}) 
        |_{X_{\infty}} 
        &
        (  p = u-r  ) \\ 
        0 & ( p  \geq u-r +1 ).
        \end{cases}
    \end{align*}

  By applying functorial flabby resolutions preserving the exactness in $\Shv (X_{\infty})$
    (e.g., $A \mapsto \prod_{x \in X_{\infty}} A_x $, 
     where $A_x$ is the stalk identified with 
        the skyscraper sheaf at $x$)
        to both 
   $\Gr^{W(M)}_{j} A^{r,*-r} $ ($j=-p-1,-p$) 
    and 
   $    W(M)_{-p} A^{r,*-r} / W(M)_{-p-2} $, 
   we can compute connecting homomorphism (\ref{eq d1 of weight-monodromy spectral sequences})
   using (the snake lemma and) 
   the differential of the total complex $s(I^{*,*})$ (\Cref{sec:sign_convention}) of 
   a double complex $(I^{*,*},d', d'')$ 
   such that 
   $I^{u,*}=0$ ($u \leq r-1$), 
   the complex $I^{u,*}$ is an injective resolution of 
   $    W(M)_{-p} A^{r,u-r} / W(M)_{-p-2} $ ($u \geq r$), 
   and 
   $s(I^{*,*})$ is an injective resolution of 
    $   W(M)_{-p} A^{r,*-r} / W(M)_{-p-2} $.
  We shall compute morphisms 
  \begin{align}\label{eq d1 separetely}
     H^{p+q} (X_{\infty}, \Gr^{W(M)}_{-p} A^{r,*-r}  ) 
     \to 
     H^{p+q+1} (X_{\infty}, \Gr^{W(M)}_{-p-1} A^{r,*-r}  ) 
  \end{align}
    given by $d'$ and $(-1)^u d'' \colon I^{u,*} \to I^{u,*+1}$, separately.

    Morphism (\ref{eq d1 separetely}) 
    given by $d'$
    is given by the differential  
   $$ -  \wedge t \colon   W(M)_{-p} A^{r,u-r} / W(M)_{-p-2}  
     \to    W(M)_{-p} A^{r,u+1-r} / W(M)_{-p-2} , $$
   which factors through 
   \begin{align*}
      W(M)_{-p} A^{r,u-r} / W(M)_{-p-2}  
       \twoheadrightarrow &
      \Gr^W_{-p+2(u-r)+1}  \F_{\X}^{u+1} (\log X_{\infty}) 
       |_{X_{\infty}}  \\
       \to &
      \Gr^W_{-p+2(u-r)+2}  \F_{\X}^{u+2} (\log X_{\infty}) 
       |_{X_{\infty}}  \\
       \hookrightarrow &
      W(M)_{-p} A^{r,u+1-r} / W(M)_{-p-2} 
       .
   \end{align*}
   Hence the morphism is exactly 
     the morphism of cohomology groups
    induced by 
     $$- \wedge t \colon
      \Gr^{W(M)}_{-p} A^{r,*-r} 
       \to \Gr^{W(M)}_{-p-1} A^{r,*-r} [1], $$
    which is the direct sum 
    \begin{align*}
      \Gr^{W(M)}_{-p} A^{r,u-r} 
      \cong  &
      \bigoplus_{\substack{ J_1 \subset I 
               \\ \# J_1 = -p + 2(u-r)+1 } }
        j_{J_1,*} \F_{X_{J_1}}^{2r-u+p}   
        \\
      \to &
      \bigoplus_{\substack{ J_2 \subset I 
              \\ \# J_2 = -p + 2(u-r)+2 } }
        j_{J_2,*} \F_{X_{J_2}}^{2r-u+p}  
       \cong  \Gr^{W(M)}_{-p-1} A^{r,u+1-r}
    \end{align*}
     of  $(-1)^{-p+2(u-r)+1} \cdot \sgn ( o(J_1), o(J_2) )  \times $
       pull-back maps
     for $J_1 \subset J_2 $.

  The other morphism (\ref{eq d1 separetely})
    given by $(-1)^u d'' \colon I^{u,*} \to I^{u,*+1}$
   coincides with 
    the connecting homomorphism of an exact sequence 
   $$ 0 \to 
     \Gr^{W(M)}_{-p-1} A^{r,u-r} [-u] 
     \to 
      W(M)_{-p} A^{r,u-r} / W(M)_{-p-2} A^{r,u-r} [-u]
     \to 
      \Gr^{W(M)}_{-p} A^{r,u-r} [-u]
      \to 0 $$
   of shifts of sheaves (with the trivial differentials).
   This can be computed in the same way as spectral sequence of tropical fan in Subsection \ref{subsec log tropical cohomology}
    (\Cref{sign of log spectral sequence:label}). 
    Since 
  $$ W(M)_{-p} A^{r,u-r} 
  := \Im (W_{-p+2(u-r) +1}\F^{u+1}_{\X } (\log X_{\infty} ) |_{ X_{\infty}  }  \to A^{r,u-r}) $$
  and 
    $$(-p + 2(u-r) +1) + (u+1-u) =p +2 (u-r+1) ,$$
   $(-1)^u d'' \colon I^{u,*} \to I^{u,*+1}$ is given by  $  (-1)^{p} \cdot \sgn (o(J'), o(J)) \times $ Gysin maps, 
   where $J' \subset J \subset I$  are as below.

  Consequently, we get the following.
   (cf. \cite[Lemme 2.7]{GuillenNavarroSurlethormelocaldescyclesinvariantsAznar1990}, 
     \cite[Subsection 11.3.2]{PetersSteenbrinkMixedHodgestructures2008}).
  \begin{prp}\label{differential of monodromy-weight spectral sequence:label}
    For $r \geq 0$, the differential 
    \begin{align*}
    d_1 \colon 
    &
    E_1^{p,q} \cong 
    \bigoplus_{ \max \{0,p\} \leq u -r }
    \bigoplus_{\substack{ J \subset I \\ \# J = -p + 2(u-r)+1 } }
    H^{2r-u+p, p+q -u}_{\Trop} (X_{J})  
    \\
     & \to 
    \bigoplus_{ \max \{0,p+1\} \leq u'-r  }
    \bigoplus_{\substack{ J \subset I \\ \# J = -p + 2(u'-r) } }
    H^{2r-u'+p+1, p +q -u'+1}_{\Trop} (X_{J}) 
     \cong E_1^{p+1,q}
    \end{align*}
    of 
    the $E_1$-terms of the spectral sequence 
    in 
    Corollary \ref{spectral sequence for ss reduction} 
    is given as the sum of 
    \begin{itemize}
      \item the sum  
    \begin{align*}
      \bigoplus_{\substack{ J_1 \subset I 
                      \\ \# J_1 = -p + 2(u-r)+1 } }
     H^{2r-u+p, p+q -u}_{\Trop} (X_{J_1})  
      \to 
      \bigoplus_{\substack{ J_2 \subset I 
                       \\ \# J_2 = -p + 2(u-r) +2 } }
      H^{2r-u+p, p +q -u}_{\Trop} (X_{J_2}) 
    \end{align*}
    of 
    $ (-1)^{p+1} \cdot  \sgn ( o(J_1), o(J_2) ) \times $
    pull-back maps  
     ($J_1 \subset J_2$,  $u+1=u'$)
    and 
    \item the sum  
     $$
    \bigoplus_{\substack{ J \subset I \\ \# J = -p + 2(u-r)+1 } }
    H^{2r-u+p, p+q -u}_{\Trop} (X_{J})  
      \to 
    \bigoplus_{\substack{ J \subset I \\ \# J = -p + 2(u-r) } }
    H^{2r-u+p+1, p +q -u+1}_{\Trop} (X_{J'}) 
    $$
    of 
     $ (-1)^{p} \cdot  \sgn (o(J'), o(J))  \times $
      Gysin maps 
    ($J' \subset J  $, $u=u'$).
    \end{itemize}
  \end{prp}

  \begin{rem}
   The difference of signs in \Cref{differential of monodromy-weight spectral sequence:label} 
   and those in   
    \cite[Lemme 2.7]{GuillenNavarroSurlethormelocaldescyclesinvariantsAznar1990} and  
     \cite[Subsection 11.3.2]{PetersSteenbrinkMixedHodgestructures2008})
     must be due to the difference of sign conventions.
     (See also \Cref{sec:sign_convention}.)
     We do not pursue this point here.
  \end{rem}

 \subsection{Eigenwaves}\label{sub:eigenwaves}
  In this subsection, we shall show that when $X$ is smooth (i.e., locally matroidal) and projective, 
  the eigenwave actions 
   are given by the \emph{reductions} of \emph{tropical Gauss-Manin connections} (\Cref{eigenwave is residue of GM connections:label}). 
  Proof is based on a theorem,  
  proved by 
   Mikhalkin-Zharkov \cite[Proof of Theorem 3]{MikhalkinZharkovTropicaleignewaveandintermediatejacobians2014} for realizable and smooth $X$
  and by Amini-Piquerez \cite[Corollary 6.21]{AminiPiquerezHodgetheoryfortropicalvarieties2020} for smooth $X$, 
  stating  
  that the eigenwave actions 
  are given by explicit morphisms $ \nu_{E_1}  $ on the spectral sequence in \Cref{spectral sequence for ss reduction}.
  We shall use the notation in the previous subsection. 
  For a while, we do not assume that $X$ is smooth nor projective.

  Our construction of tropical Gauss-Manin connections 
   uses connecting homomorphisms of relative tropical homolorphic forms,
   and is an analog of a construction of Katz-Oda (\cite{KatzOdaOnthedifferentiationofDeRhamcohomologyclasseswithrespecttoparameters1968}).
  Recall that 
   $\Psi \colon \Trop (T_{\tilde{\Lambda} |_{ N_\R \times \R }})\to \R \cup \{ \infty \} $
   is 
   the natural extension of the projection 
   $ N_\R \times \R \to \R$.
  Recall that 
   by \Cref{Fp log computation}, 
     there is an exact sequenece 
    \begin{align*}
    0 \to \F^1_{\R \cup \{ \infty  \}} (\log \{ \infty  \}) 
     \wedge \F^{r-1}_\X (\log X_{\infty}) |_{ \X^{\circ} }
       \to &  \F^r_\X (\log X_{\infty}) |_{ \X^{\circ} } \\
       \to & \F^1_{\R \cup \{ \infty  \}} (\log \{ \infty  \}) 
     \wedge \F^{r}_\X (\log X_{\infty}) |_{ \X^{\circ} }
       \to 0 . 
    \end{align*}
  We have put 
   $$\F^r_{\X / \R \cup \{ \infty  \} } (\log X_{\infty} ) 
   :=
    \F^r_\X (\log X_{\infty}) / 
    \F^1_{\R \cup \{ \infty  \}} (\log \{ \infty  \}) 
    \wedge \F^{r-1}_\X (\log X_{\infty}).$$
  We immediately get an exact sequence 
    \begin{align*}
    0 \to & \Psi^* \F^1_{\R \cup \{ \infty  \}} (\log \{ \infty  \}) 
     \otimes 
      \F^{r-1}_{\X / \R \cup \{ \infty  \} } (\log X_{\infty} )  [-r]
      |_{ \X^{\circ} } \\
       \to &  \F^r_\X (\log X_{\infty}) [-r] |_{ \X^{\circ} } 
       \to  
      \F^{r}_{\X / \R \cup \{ \infty  \} } (\log X_{\infty} ) 
      [-r]
       |_{ \X^{\circ} }
       \to 0 . 
    \end{align*}
   By pushing forward this exact sequence 
   under $\Psi$
   and  
    identifing 
    $ \F^1_{\R \cup \{ \infty  \}} (\log \{ \infty  \}) |_{\R_{>0 } \cup \{ \infty \}}$
    with the constant sheaf $\underline{\Q}_{ \R_{>0 } \cup \{ \infty \} }$
    by $t \mapsto 1$,
    we get 
  the connecting homomorhism 
  $$ \res (\nabla) 
     \colon 
     R^q\Psi_*  (
      \F^{r}_{\X / \R \cup \{ \infty  \} } (\log X_{\infty} ) 
      [-r]
      |_{ \X^{\circ} })
      \to 
     R^{q+1}\Psi_* (\F^{r-1}_{\X / \R \cup \{ \infty  \} } (\log X_{\infty} ) 
     [-r]
       |_{ \X^{\circ} } ) .$$
   We call this the \emph{residue} of the \emph{tropical Gauss-Manin connection}
        (cf. \cite{KatzOdaOnthedifferentiationofDeRhamcohomologyclasseswithrespecttoparameters1968}). 
        Recall that 
    there are an isomorphism 
   $$\F^{r}_{\X / \R \cup \{ \infty  \} } (\log X_{\infty} ) 
       |_{ X \times \{1\} } \cong 
      \F^{r}_{X } $$
   and isomorphisms 
  \begin{align*}
     H^{r,s}_{\Trop } (X )
     &
     \cong  H^{s+r} ( \X^{\circ} , 
    \F^r_{\X / \R \cup \{ \infty  \} } (\log X_{\infty} ) [ -r ] 
     |_{ \X^{\circ} })  \\ 
     &
   \cong H^{s+r} ( X_{\infty} ,  
   \F^r_{\X / \R \cup \{ \infty  \} } (\log X_{\infty} ) [-r] |_{ X_{\infty}  })
  \end{align*}
   (Subsection \ref{sub:monodromy_weight_spectral_sequences}).
        Thus (using any of the above three cohomology) we get a morphism 
        \begin{align*}
      \res (\nabla)
     \colon 
     H^{r,s}_{\Trop } (X ) \to 
     H^{r-1,s+1}_{\Trop } (X )
        \end{align*}
        of tropical cohomology.
  We shall show that 
   this morphism coincides with the
   eigenwave action (up to sign) when $X$ is smooth. 

   Following Steenbrink \cite[(4.22)]{SteenbrinkLimitsofHodgestructures1976}, 
   \cite[Theorem 11.21]{PetersSteenbrinkMixedHodgestructures2008}, 
   we shall give a description of $\res (\nabla)$
   in terms of the resolution $A^{r,*-r}$.
  We put 
   $$\xymatrix{
     A^{r,q-r} \ar@{=}[d] \ar[r]^-{\nu} & 
        A^{r-1,q-r+1} \ar@{=}[d]  \\ 
     \F^{q+1}_{\X} (\log X_{\infty} )/ W_{q-r} \F^{q+1}_{\X} (\log X_{\infty} ) |_{X_{\infty}}
    \ar@{->>}[r]  &
     \F^{q+1}_{\X} (\log X_{\infty} )/ W_{q-r+1} \F^{q+1}_{\X} (\log X_{\infty} ) |_{X_{\infty}}
   }$$
   the projection. 
  This morphism induces a morphism 
  $\nu \colon A^{r,*-r} \to A^{r-1, *-r+1}$
   of complexes, 
   and hence a morphism of cohomology groups 
   $$\nu \colon H^{r,s}_{\Trop }(X) \cong 
          H^{s+r} (X_0, A^{r,*-r}) 
      \to  H^{s+r} (X_0, A^{r-1,*-r+1}) 
      \cong H^{r-1,s+1}_{\Trop }(X).$$

  The following is an analog of
    \cite[(4.22)]{SteenbrinkLimitsofHodgestructures1976}, 
   \cite[Theorem 11.21]{PetersSteenbrinkMixedHodgestructures2008}.  
  \begin{prp}
   We have a commutative diagram 
   $$\xymatrix{
     H^{s+r} ( X_{\infty} ,  
   \F^r_{\X / \R \cup \{ \infty  \} } (\log X_{\infty} ) [-r] |_{ X_{\infty}  })
   \ar[d]^-{\cong}
      \ar[r]^-{\res (\nabla) }  & 
      H^{s+r} ( X_{\infty}   ,   
   \F^{r-1}_{\X / \R \cup \{ \infty  \} } (\log X_{\infty} ) [-r+1] |_{ X_{\infty}  }) 
     \ar[d]^-{\cong}
     \\
          H^{s+r} (X_0, A^{r,*-r})   \ar[r]^-{\nu}  & 
       H^{s+r} (X_0, A^{r-1,*-r+1}) .
   }$$
  \end{prp}
  \begin{proof}
   We have exact sequences 
    \begin{align*}
    0 \to & 
      \F^{r-1}_{\X / \R \cup \{ \infty  \} } (\log X_{\infty} )  [-r+1][-1]
      |_{ X_{\infty} } \\
       \to &  \F^r_\X (\log X_{\infty}) [-r] |_{ X_{\infty} } 
       \to  
      \F^{r}_{\X / \R \cup \{ \infty  \} } (\log X_{\infty} ) 
      [-r]
       |_{ X_{\infty} }
       \to 0  
    \end{align*}
    and 
    $$ 0 \to A^{r-1,*-r+1}[-1]
         \to  \cone (\nu)[-1]
         \to A^{r,*-r}
    \to 0 $$
    (see \Cref{sec:sign_convention} for our sign convention of mapping cones), 
    which are 
    compatible with 
    quasi-isomorphisms
  $$\F^u_{\X / \R \cup \{ \infty  \} } (\log X_{\infty} ) [-u] |_{ X_{\infty}  } 
  \ni \alpha 
  \mapsto  \alpha \wedge t
   \in  A^{u,*-u}$$
   ($u=r-1,r$)
   and 
   $$ \F^r_\X (\log X_{\infty}) [-r] |_{ X_{\infty} }  
      \ni \alpha 
      \mapsto (-\alpha \wedge t, \alpha)
     \in   \cone (\nu)[-1].
   $$
   The morphism $\res (\nabla) $ is the connecting homomorphism of the first exact sequence, which is compatible with that of the second one.
   Hence the assertion holds.
  \end{proof}

  The morphism
  $\nu $ induces a morphism 
  $ \nu_{E_2} \colon E_{2, A^{r,*-r}}^{p,q}  
  \to 
   E_{2, A^{r-1,*-r+1}}^{p+2,q-2}  
  $ of $E_2$-terms. This can be seen in the following way.
  The morphism $\nu$ maps $W(M)_{-p} A^{r,*-r}$ to $W(M)_{-p-2} A^{r-1,*-r+1} $, 
   and 
    hence it induces a morphism
  $$ \nu_{E_1} \colon E_{1, A^{r,*-r}}^{p,q} = 
          H^{p+q} (X_0, \Gr^{W(M)}_{-p} A^{r,*-r}  ) 
  \to 
   E_{1, A^{r-1,*-r+1}}^{p+2,q-2} = 
         H^{p+q} (X_0, \Gr^{W(M)}_{-p-2} A^{r-1,*-r+1}  ) 
  $$
    of $E_1$-terms
    of the spectral sequences.
  More presicely, it is induced by 
  the sum
  \begin{align*}
   \nu \colon 
   \Gr^{W(M)}_{-p} A^{r,*-r} 
   \cong &  
    \bigoplus_{ \max \{0, p \} \leq u-r  }
    \bigoplus_{\substack{ J \subset I \\ \# J = -p + 2(u-r)+1 } }
     j_{J,*} \F_{X_{J}}^{2r-u+p}  [-u]  \\
      \to  & 
    \bigoplus_{ \max \{0, p +2 \} \leq u' -r+1 }
    \bigoplus_{\substack{ J \subset I \\ \# J = -p + 2(u'-r)+1 } }
     j_{J,*} \F_{X_{J}}^{2r-u'+p}  [-u'] \\
   \cong &  
   \Gr^{W(M)}_{-p-2} A^{r-1,*-r+1}  
  \end{align*} 
  of the identitiy map 
  of $j_{J,*} \F_{X_{J}}^{2r-u+p}  [-u]$
  for $\max \{0,p+1 \} \leq u -r$, $u=u'$ and $J \subset I $ with $ \# J = -p + 2u+1 $
  (where the order of $J\subset I$
    used to give the residue maps are the same).
  This is obviously compatible with $d_1$, and hence induces a morphism 
  $ \nu_{E_2} \colon E_{2, A^{r,*-r}}^{p,q}  
  \to 
   E_{2, A^{r-1,*-r+1}}^{p+2,q-2} . 
  $

  We assume that $X$ is smooth, i.e., locally matroidal and $T_{\Sigma}$ is projective.
   Recall that 
   $X_J$ ($J \subset I$) is the natural compactification of the star fan of a polyhedron in
   $N$-unimodular $\Q$-rational 
    polyhedral complex strucuture $\Lambda$ of $X$.
     Hence $X_J$ is also smooth, 
   and by \cite[Theorem 1.1]{AminiPiquerezTropicalFeichtner-Yuzvinskyandpositivitycriterionforfans2024} and \cite[Theorem 2]{JellShawSmackaSuperformstropicalcohomologyandPoincarduality2019}, 
    we have $H^{p,q}_{\Trop} (X_J) =0$ for $p \neq q$.
   Hence $E_{1, A^{r,*-r}}^{p,q} \neq 0 $ only if $q=2r$, 
   and 
   we have 
   $E_{2, A^{r,*-r}}^{p,2r} \cong H^{r,p+r}_{\Trop }(X)$. 
   Then 
   the morphism
   $\nu  \colon H^{r,p+r}_{\Trop }(X)
     \to  H^{r-1,p+r+1}_{\Trop }(X)$ 
   coincides with the morphism $ \nu_{E_2}  $ on $E_2$-terms.
   Mikhalkin-Zharkov \cite[Proof of Theorem 3]{MikhalkinZharkovTropicaleignewaveandintermediatejacobians2014} (for realizable $X$)
  and Amini-Piquerez \cite[Corollary 6.21]{AminiPiquerezHodgetheoryfortropicalvarieties2020}(in general)
  proved that the eigenwave action coincides with 
   $ \nu_{E_2} $ (up to sign).

   Consequently, we have the following.
  \begin{thm}\label{eigenwave is residue of GM connections:label}
   We assume that $X$ is smooth and $T_{\Sigma}$ is projective.
   Then the residue 
   $$ \res (\nabla)  
      \colon 
     H^{r,s}_{\Trop} ( X)
      \to   
      H^{r-1,s+1}_{\Trop} ( X) 
   $$
    of tropical Gauss-manin connection
   coincides with the eigenwave action, up to sign.
  \end{thm}

\appendix
\section{The sign convention} \label{sec:sign_convention}
   Signs in homological algebra are very delicate, see e.g.,  
   \cite{ConradGrothendieckDualityandBaseChange2000},
    \cite{LawsonInwhichItrytogetthesignsrightforonce2013}. 
   In this appendix, 
   we clarify our sign convention.
   For homological algebra, see \cite{WeibelAnintroductiontohomologicalalgebra1984}. 
   
  Let $\mathscr{A}$ be the category of abelian sheaves on a topological space,
   and $C(\mathscr{A})$ the category of complexes in $\mathscr{A}$.
  Let $(A,d_A)$ and $(B,d_B)$ be complexes in $C(\mathscr{A})$.
  For  $r \in \Z$,
   we put $A[r]$ the complex defined by
   $A[r]^p:= A^{p+r}$ with $d_{A[r]} := (-1)^{r} d_A$.
  We have a canonical identification $H^{p+r}(A) = H^{p}(A[r])$ 
   given by the identity map of $A^{p+r}$.
  For a morphism $f \colon A \to B$ of complexes, 
   we put  $f[r] \colon A[r] \to B[r]$ 
   the morphism given by $f[r]^p := f^{p+r}$.

  We put $A \otimes B$ the complex 
   such that 
   $$ (A \otimes B)^p := \bigoplus_{i+j =p} A^i \otimes B^j$$
   and 
   \begin{align}\label{eq sign of tensor product}
    d_{A \otimes B} |_{ A^i \otimes B^j }
      := d_A \otimes \id_B + \id_A \otimes (-1)^i d_B   .
   \end{align}
   We have a natural isomorphism 
   $$ (A[r]) \otimes (B[s]) \cong (A \otimes B) [r+s]$$
   given by 
   $(-1)^{(i-r) s}$ on $A^i \otimes B^j$ (\cite[p.11 and p.15]{ConradGrothendieckDualityandBaseChange2000}).
   We have an isomorphism 
     $$ A \otimes B \cong B \otimes A$$
     given by 
     $ a \otimes b \mapsto (-1)^{pq} b \otimes a$
     ($a \in A^p $, $b \in B^q$).

   A double complex $B^{p,q}$
    is equipped with morphisms
     $d' \colon B^{p,q} \to B^{p+1, q}$
     and 
     $d'' \colon B^{p,q} \to B^{p, q+1}$
     such that 
     $d' \circ d' = 0$,
     $d'' \circ d'' = 0$, and 
     $d' \circ d'' = d'' \circ d' $.
  The associated single complex $s(B^{ *,* })$
  is defined by 
    $ s(B^{*,*})^r := \bigoplus_{ p+q = r } B^{p,q}$
    and the sum of 
     $d' \colon B^{p,q} \to B^{p+1, q}$
     and 
     $(-1)^p \cdot d'' \colon B^{p,q} \to B^{p, q+1}$.

  Our signs of mapping cones and cylinders 
    are the same as \cite{ConradGrothendieckDualityandBaseChange2000}, \cite{LawsonInwhichItrytogetthesignsrightforonce2013}, 
    which is different from \cite{WeibelAnintroductiontohomologicalalgebra1984}. 
  For a morphism $f \colon A \to B$ of complexes, 
   the mapping cone is
   $\cone (f)^n:= A^{n+1}\oplus B^n$ 
   with 
   $$d(a,b):= (-d_A (a), d_B (b)+ f(a)).$$
   There are natural morphisms 
   \begin{align*}
    B \ni & b \mapsto (0,b) \in \cone (f) ,  \\
    \cone ( f ) \ni & (a,b) \mapsto - a \in A[1] .
   \end{align*}
   The mapping cylinder is 
   defined as
   $\cyl (f)^n := 
     A^n \oplus A^{n+1} \oplus B^n $
    with $$d(a,a',b):= 
          (d_a(a)-a', 
             -d_A (a'), d_B (b) + f (a') ).$$
   There are natural morphisms 
   \begin{align*}
    A \ni & a \mapsto (a,0,0) \in \cyl (f) , \\
    \cyl ( f ) \ni & (a,a', b) \to (a',b) \in \cone (f) .
   \end{align*}

  By this sign convention of mapping cylinders and mapping cones, 
  for 
    a short exact sequence 
  $$ 0 \to A \to B \to C \to 0$$
    in $C(\mathscr{A})$, 
   connecting homomorphisms $ H^i(C) \to H^{i+1}(A)$
   of the snake lemma 
   and morphisms 
   $ H^i(C) \to H^{i+1}(A)$
    (\cite[Section 1.5 and Definition 10.1.3]{WeibelAnintroductiontohomologicalalgebra1984})
   of cohomology groups of 
   the distinguished triangle
   $$A \to B \to C \to^{[1]}$$ 
   in the derived category $D(\mathscr{A})$ 
   is the same 
   (see  \cite[Section 1.3]{ConradGrothendieckDualityandBaseChange2000}).
   (If we use 
   the sign convention of 
   \cite{WeibelAnintroductiontohomologicalalgebra1984}, 
   two morphisms differ by $(-1)$ 
   (see \cite[Exercise 1.5.6]{WeibelAnintroductiontohomologicalalgebra1984}).)

\bibliographystyle{amsalpha}
\bibliography{mikamibibtex}

@misc{MikamiDifferentialformsandcohomologyintropicalandcomplexgeometry2021,
	author = {R. Mikami},
	howpublished = {arXiv:2106.11479},
	title = {Differential forms and cohomology in tropical and complex geometry},
	year = {2021}}

@article{AksnesAminiPiquerezShawCOHOMOLOGICALLYTROPICALVARIETIES2025,
	author = {E. Aksnes and O. Amini and M. Piquerez and K. Shaw},
	journal = {J. Inst. Math. Jussieu},
	number = {6},
	pages = {2543-2572},
	title = {COHOMOLOGICALLY TROPICAL VARIETIES},
	volume = {24},
	year = {2025}}

@article{KatzTHEREGULARITYTHEOREMINALGEBRAICGEOMETRY1970,
	author = {N. M. Katz},
	journal = {Actes. Congr\`{e}s intern. math.},
	pages = {437-443},
	title = {THE REGULARITY THEOREM IN ALGEBRAIC GEOMETRY},
	year = {1970}}

@incollection{JellTROPICALCOHOMOLOGYWITHINTEGRALCOEFFICIENTSFORANALYTICSPACES2022,
	author = {P. Jell},
	booktitle = {Facets of Algebraic Geometry: A Collection in Honor of {W}illiam {F}ulton's 80th Birthday},
	publisher = {Cambridge University Press},
	series = {London Mathematical Society Lecture Note Series},
	title = {TROPICAL COHOMOLOGY WITH INTEGRAL COEFFICIENTS FOR ANALYTIC SPACES},
	year = {2022}}

@article{LiuMonodromymapfortropicalDolbeaultcohomology2019,
	author = {Y. Liu},
	journal = {Algebr. Geom.},
	number = {4},
	pages = {384-409},
	title = {Monodromy map for tropical {D}olbeault cohomology},
	volume = {6},
	year = {2019}}

@article{DeligneThoriedeHodgeII1972,
	author = {P. Deligne},
	journal = {Publ. Math. IHES},
	pages = {5-57},
	title = {Th{{\'e}}orie de {H}odge {II}},
	volume = {40},
	year = {1972}}

@article{KatzOdaOnthedifferentiationofDeRhamcohomologyclasseswithrespecttoparameters1968,
	author = {N. M. Katz and T. Oda},
	journal = {J. Math. Kyoto Univ.},
	number = {2},
	pages = {199-213},
	title = {On the differentiation of {D}e {R}ham cohomology classes with respect to parameters},
	volume = {8},
	year = {1968}}

@misc{YamamotoTROPICALCONTRACTIONSTOINTEGRALAFFINEMANIFOLDSWITHSINGULARITIES2021,
	author = {Y. Yamamoto},
	howpublished = {arXiv21051014},
	title = {TROPICAL CONTRACTIONS TO INTEGRAL AFFINE MANIFOLDS WITH SINGULARITIES},
	year = {2021}}

@article{ZharkovTorusfibrationsofCalabi-Yauhypersurfacesintoricvarieties2000,
	author = {I. Zharkov},
	journal = {Duke Math. J.},
	number = {2},
	pages = {237-257},
	title = {Torus fibrations of {C}alabi-{Y}au hypersurfaces in toric varieties},
	volume = {101},
	year = {2000}}

@article{GrossSpecialLagrangianfibrations.I.Topology1998,
	author = {M. Gross},
	journal = {Integrable Systems and Algebraic Geometry (Kobe/Kyoto, 1997) (World Scientific Publishing, River Edge)},
	pages = {156-193},
	title = {Special {L}agrangian fibrations. {I}. Topology},
	year = {1998}}

@article{SteenbrinkLimitsofHodgestructures1976,
	author = {J. Steenbrink},
	journal = {Invent. Math.},
	pages = {229-257},
	title = {Limits of {H}odge structures},
	volume = {31},
	year = {1976}}

@book{KempfKnudsenMumfordSaint-DonatToroidalEmbeddings11973,
	author = {G. Kempf and F. Knudsen and D. Mumford and B. Saint-Donat},
	publisher = {Springer Berlin, Heidelberg},
	series = {Lecture Notes in Mathematics},
	title = {Toroidal Embeddings 1},
	year = {1973}}

@misc{SmackaPhD2017,
	author = {J. Smacka},
	howpublished = {Ph.D. Thesis 2017, Available at https:// epub .uni -regensburg .de /36262/},
	title = {Differential forms on tropical spaces},
	year = {2017}}

@misc{MikamiTropicalIntersectionHomology2024,
	author = {R. Mikami},
	howpublished = {arXiv:2412.20748},
	title = {Tropical intersection homology},
	year = {2024}}

@book{ConradGrothendieckDualityandBaseChange2000,
	author = {B. Conrad},
	publisher = {Springer Berlin, Heidelberg},
	series = {Lecture Notes in Mathematics},
	title = {{G}rothendieck Duality and Base Change},
	year = {2000}}

@misc{LawsonInwhichItrytogetthesignsrightforonce2013,
	author = {T. Lawson},
	howpublished = {available at https://www-users.cse.umn.edu/~tlawson/papers/signs.pdf},
	title = {In which {I} try to get the signs right for once},
	year = {2013}}

@article{GuillenNavarroSurlethormelocaldescyclesinvariantsAznar1990,
	author = {F. Guill{\'e}n and V. Navarro-Aznar},
	journal = {Duke Math. J.},
	number = {1},
	pages = {133-155},
	title = {Sur le th{\'e}or{\`e}me local des cycles invariants},
	volume = {61},
	year = {1990}}

@misc{AminiPiquerezHodgetheoryfortropicalvarieties2020,
	author = {O. Amini and M. Piquerez},
	howpublished = {arXiv:2007.07826},
	title = {Hodge theory for tropical varieties},
	year = {2020}}

@book{PetersSteenbrinkMixedHodgestructures2008,
	author = {C.A.M. Peters and J.H.M. Steenbrink},
	publisher = {Springer-Verlag, Berlin},
	series = {Ergebnisse der Mathematik und ihrer Grenzgebiete. 3. Folge. A Series of Modern Surveys in Mathematics [Results in Mathematics and Related Areas. 3rd Series. A Series of Modern Surveys in Mathematics]},
	title = {Mixed {H}odge structures},
	volume = {52},
	year = {2008}}

@article{GrossSiebertMirrorsymmetryvialogarithmicdegenerationdataII2010,
	author = {M. Gross and B. Siebert},
	journal = {J. Algebraic Geom.},
	number = {4},
	pages = {679-780},
	title = {Mirror symmetry via logarithmic degeneration data, {II}},
	volume = {19},
	year = {2010}}

@article{RabinoffTropicalanalyticgeometryNewtonpolygonsandtropicalintersections2012,
	author = {J. Rabinoff},
	journal = {Adv.Math.},
	number = {6},
	pages = {3192-3255},
	title = {Tropical analytic geometry, {N}ewton polygons, and tropical intersections},
	volume = {229},
	year = {2012}}

@article{PayneAnalytificationisthelimitofalltropicalizations2009,
	author = {S. Payne},
	journal = {Math. Res. Lett.},
	number = {3},
	pages = {543-556},
	title = {Analytification is the limit of all tropicalizations},
	volume = {16},
	year = {2009}}

@book{MaclaganSturmfelsIntroductiontotropicalgeometry2015,
	author = {D. Maclagan and B. Sturmfels},
	publisher = {American Mathematical Society, Providence, RI},
	series = {Graduate Studies in Mathematics},
	title = {Introduction to tropical geometry},
	volume = {161},
	year = {2015}}

@conference{GublerAguidetotropicalizations2013,
	author = {W. Gubler},
	booktitle = {Algebraic and combinatorial aspects of tropical geometry},
	pages = {125-189},
	publisher = {Amer. Math. Soc., Providence, RI},
	series = {Contemp. Math.},
	title = {A guide to tropicalizations},
	volume = {589},
	year = {2013}}

@article{TotaroChowgroupsChowcohomologyandlinearvarieties2014,
	author = {B. Totaro},
	journal = {Forum Math. Sigma 2},
	pages = {Paper No. e17, 25},
	title = {{C}how groups, {C}how cohomology, and linear varieties},
	year = {2014}}

@misc{AminiPiquerezHomologicalsmoothnessandDeligneresolutionfortropicalfans2024,
	author = {O. Amini and M. Piquerez},
	howpublished = {arXiv:2105.01504},
	title = {Homological smoothness and {D}eligne resolution for tropical fans},
	year = {2024}}

@misc{MikamiOntropicalcycleclassmaps2020,
	author = {R. Mikami},
	howpublished = {arXiv:2009.04690},
	title = {On tropical cohomology of smooth algebraic varieties},
	year = {2020}}

@misc{AminiPiquerezTropicalFeichtner-Yuzvinskyandpositivitycriterionforfans2024,
	author = {O. Amini and M. Piquerez},
	howpublished = {arXiv:2405.05014},
	title = {Tropical {F}eichtner-{Y}uzvinsky and positivity criterion for fans},
	year = {2024}}

@article{JellShawSmackaSuperformstropicalcohomologyandPoincarduality2019,
	author = {P. Jell and K. Shaw and J. Smacka},
	journal = {Adv. Geom.},
	number = {1},
	pages = {101-130},
	title = {Superforms, tropical cohomology, and {P}oincar{\'e} duality},
	volume = {19},
	year = {2019}}

@article{ItenbergKatzarkovMikhalkinZharkovTropicalhomology2019,
	author = {I. Itenberg and L. Katzarkov and G. Mikhalkin and I. Zharkov},
	journal = {Math. Ann.},
	number = {1-2},
	pages = {963-1006},
	title = {Tropical homology},
	volume = {374},
	year = {2019}}

@conference{MikhalkinZharkovTropicaleignewaveandintermediatejacobians2014,
	author = {G. Mikhalkin and I. Zharkov},
	booktitle = {Homological mirror symmetry and tropical geometry},
	pages = {309-349},
	publisher = {Springer, Cham},
	series = {Lect. Notes Unione Mat. Ital.},
	title = {Tropical eigenwave and intermediate jacobians},
	volume = {15},
	year = {2014}}

@book{CoxaLittleSchenckToricvarieties2011,
	author = {D. A. Cox and J. B. Little and H. K. Schenck},
	publisher = {American Mathematical Society, Providence, RI},
	series = {Graduate Studies in Mathematics},
	title = {Toric varieties},
	volume = {124},
	year = {2011}}

@book{WeibelAnintroductiontohomologicalalgebra1984,
	author = {C. A. Weibel},
	publisher = {Cambridge University Press, Cambridge},
	series = {Cambridge Studies in Advanced Mathematics},
	title = {An introduction to homological algebra},
	volume = {38},
	year = {1994}}

@article{GrossShokirehAsheaf-theoreticapproachtotropicalhomology23,
	author = {A. Gross and F. Shokrieh},
	journal = {J. Algebra},
	pages = {577-641},
	title = {A sheaf-theoretic approach to tropical homology},
	volume = {635},
	year = {2023}}

@misc{AminiPiquerezHomologyoftropicalfans2021,
	author = {O. Amini and M. Piquerez},
	howpublished = {arXiv:2105.01504},
	title = {Homology of tropical fans},
	year = {2021}}

\end{document}